\documentclass[12pt]{amsart}
\usepackage[utf8]{inputenc}

\usepackage{amssymb}
\usepackage{bm}
\usepackage{graphicx}
\usepackage[centertags]{amsmath}
\usepackage{amsfonts}
\usepackage{amsthm}
\usepackage{enumerate}
\usepackage{tocvsec2}
\usepackage{xcolor}
\usepackage[margin=1in]{geometry}

\newtheorem{theorem}{Theorem}[section]
\newtheorem*{theorem*}{Theorem}

\newtheorem{corollary}[theorem]{Corollary}
\newtheorem{lemma}[theorem]{Lemma}

\newtheorem{proposition}[theorem]{Proposition}

\newtheorem{claim}[theorem]{Claim}

\theoremstyle{definition}

\newenvironment{remark}[1][Remark]{\begin{trivlist}
\item[\hskip \labelsep {\bfseries #1}]}{\end{trivlist}}

\newcommand{\rr}{\mathbb{R}}
\newcommand{\nn}{\mathbb{N}}

\newcommand{\ee}{\varepsilon}

\newcommand{\hhh}{\mathcal{H}}

\newcommand{\supp}{\mathrm{supp}}

\newcommand{\cat}{^\smallfrown}
\newcommand{\ord}{\textbf{Ord}}
\newcommand{\gr}{\Gamma}
\newcommand{\pr}{\mathbb{P}}

\newcommand{\tens}{X\hat{\otimes}_\ee Y}

\begin{document}

\title[Szlenk index of injective tensor products]{The Szlenk index of injective tensor products and convex hulls}
\author{Ryan M. Causey}

\address{Department of Mathematics\\ Miami University\\
Oxford, OH 45056\\ U.S.A.} \email{causeyrm@miamioh.edu}

\begin{abstract} Given any Banach space $X$ and any weak*-compact subset $K$ of $X^*$, we compute the Szlenk index of the weak*-closed, convex hull of $K$ as a function of the Szlenk index of $K$. Also as an application, we compute the Szlenk index of any injective tensor product of two operators. In particular, we compute the Szlenk index of an injective tensor product $\tens$ in terms of $Sz(X)$ and $Sz(Y)$.   As another application, we give a complete characterization of those ordinals which occur as the Szlenk index of a Banach space, as well as those ordinals which occur as the Bourgain $\ell_1$ or $c_0$ index of a Banach space.

\end{abstract}

\maketitle

\section{introduction}

Since Szlenk introduced his index \cite{Szlenk} to prove the non-existence of a separable, reflexive Banach space which is universal for the class of separable, reflexive Banach spaces, the Szlenk index has become an important tool in Banach space theory. For a survey of these results, we refer the reader to \cite{Lancien2}.  One remarkable property of the Szlenk index is that it perfectly determines the isomorphism classes of separable $C(K)$ spaces, which can be seen by combining the results of Bessaga and Pe\l czy\'{n}ski \cite{BP}, Milutin \cite{Milutin}, and Samuel \cite{Samuel}. In the sequel, we shall use $w^*$ in place of weak$^*$ to refer to the weak$^*$ topology on a dual Banach space.  The Szlenk index is also an important tool regarding asymptotically uniformly smooth and $w^*$-UKK renormings \cite{KOS}, and $\xi$-$w^*$-UKK renormings of separable Banach spaces \cite{LPR}.  For these questions, the relationship between the Szlenk index of a given set and the Szlenk index of its $w^*$-closed, convex hull is important.  The content of this work provides a sharp result regarding the relationship between these two indices.

Let us recall the definition of the Szlenk derivation and the Szlenk index.  Given a Banach space $X$, a $w^*$-compact subset $K$ of $X^*$, and $\ee>0$, we let $s_\ee(K)$ consist of all those $x^*\in K$ such that for all $w^*$-neighborhoods $V$ of $x^*$, $\text{diam}(V\cap K)>\ee$.    We define the transfinite derived sets by $s^0_\ee(K)=K$, $s^{\xi+1}_\ee(K)=s_\ee(s^\xi_\ee(K))$, and $s^\xi_\ee(K)=\cap_{\zeta<\xi}s^\zeta_\ee(K)$ when $\xi$ is a limit ordinal.  We let $Sz_\ee(K)=\min\{\xi: s_\ee^\xi(K)=\varnothing\}$ if this class of ordinals is non-empty, and we write $Sz_\ee(K)=\infty $ otherwise. We agree to the convention that $\xi<\infty$ for any ordinal $\xi$.   We let $Sz(K)=\sup_{\ee>0}Sz_\ee(K)$, with the convention that the supremum is $\infty$ if $Sz_\ee(K)=\infty$ for some $\ee>0$.  Last, we let $Sz(X)=Sz(B_{X^*})$ for a Banach space $X$.     Given an operator $A:X\to Y$, we let $Sz(A)=Sz(A^*B_{Y^*})$.

We also recall the definition of the Cantor-Bendixson derivative of a subset of a topological space.  If $K$ is a topological space and $L\subset K$, we let $L'$ consist of all members of $L$ which are not isolated in $L$.   We define the transfinite derived sets $L^\xi$ for all ordinals $\xi$ as above: $L^0=L$, $L^{\xi+1}=(L^\xi)'$, and $L^\xi=\cap_{\zeta<\xi}L^\zeta$ when $\xi$ is a limit ordinal.    We let $i(L)$ denote the minimum $\xi$ such that $L^\xi=\varnothing$ if such a $\xi$ exists, and $i(L)=\infty$ otherwise.   We recall that $L$ is \emph{scattered} if and only if $i(L)<\infty$.

Throughout, $\ord$ denotes the class of ordinal numbers, $\nn=\{1, 2, \ldots\}$, $\nn_0=\{0\}\cup \nn$. We denote the first infinite ordinal by $\omega$ and the first uncountable ordinal by $\omega_1$.  Recall that a gamma number is an ordinal which is greater than the sum of any two lesser ordinals.   Of course, $0$ is a gamma number.  The non-zero gamma numbers are precisely the ordinals of the form $\omega^\xi$ for some $\xi$ \cite{Monk}.   Given an ordinal $\xi$, we let $\Gamma(\xi)$ denote the minimum gamma number which is not less than $\xi$.  Since $\omega^\xi\geqslant \xi$ for any ordinal $\xi$, this minimum exists.   For completeness, we agree that $\Gamma(\infty)=\infty$.  For the convenience, we will also agree to call $\infty$ a gamma number.   

The main tool of this work is the following.  

\begin{theorem} Let $X$ be any Banach space, $K\subset X^*$ $w^*$-compact, and let $L=\overline{\emph{co}}^*(K)$.  Then $Sz(L)=\Gamma(Sz(K))$.   
\label{theorem}
\end{theorem}

Using the geometric version of the Hahn-Banach theorem (or in some cases the Krein-Milman theorem), Theorem \ref{theorem} will allow us to  compute the Szlenk index of a given set $L$ by the Szlenk index of a subset of $L$ whose $w^*$-closed, convex hull is $L$ and whose Szlenk index is easy to compute. Our first example of such a phenomenon is the following.

\begin{theorem} Let $K$ be any compact, Hausdorff topological space.  Then $Sz(C(K))=\Gamma(i(K))$.

\label{maintheorem}
\end{theorem}

Recall that for two Banach spaces $X, Y$, $X\hat{\otimes}_\ee Y$ denotes the injective tensor product of $X$ and $Y$.  All relevant definitions will be given in the next section.  Recall also that for operators $A:X\to X_1$, $B:Y\to Y_1$, there is an induced operator from $X\hat{\otimes}_\ee Y$ into $X_1\hat{\otimes}_\ee Y_1$ which is the continuous, linear extension of the operator $x\otimes y\mapsto Ax\otimes By$. We denote this operator by $A\otimes B$.  Using Theorem \ref{theorem} and known results about the Szlenk index, we will prove the following. 

\begin{theorem} Let $A:X\to X_1$ $B:Y\to Y_1$ be non-zero operators.  Then $$Sz(A\otimes B)=\max\{Sz(A), Sz(B)\}.$$    In particular, for any non-zero Banach spaces $X,Y$, $Sz(\tens)=\max\{Sz(X), Sz(Y)\}$.

\label{maintheorem3}
\end{theorem}

If $K$ is compact, Hausdorff and $X$ is a Banach space, we let $C(K, X)$ denote the continuous, $X$-valued functions defined on $K$.  By the work of Grothendieck \cite{Grothendieck} (see also \cite{Ryan}), $C(K, X)=C(K)\hat{\otimes}_\ee X$.  Using this fact together with Theorems \ref{maintheorem} and  \ref{maintheorem3}, we obtain the following.

\begin{theorem} Let $K$ be compact, Hausdorff and let $X$ be a non-zero Banach space.  Then  $Sz(C(K,X))=\max\{Sz(C(K)), Sz(X)\}$.   

\label{maintheorem2}
\end{theorem}

Combining Theorems \ref{maintheorem} and \ref{maintheorem2} completely determines the Szlenk index of $C(K,X)$ in terms of the Szlenk index of $X$ and the Cantor-Bendixson index of $K$.

Theorem \ref{theorem}allows us to easily construct a Banach space with prescribed Szlenk index if we can construct a norming set for that Banach space which has this Szlenk index. Recall that the \emph{cofinality} of an ordinal $\xi$ is the minimum cardinality of a subset $A$ of $[0, \xi)$ such that $\sup A=\sup [0, \xi)$.  

\begin{theorem} Let $\Gamma=\{\omega^\xi: \xi\in \textbf{\emph{Ord}}, \xi\text{\ \emph{has countable cofinality}}\}$ and $\Lambda=\{\omega^{\omega^\xi}: \xi\in \textbf{\emph{Ord}}, \xi\text{\ \emph{is a limit ordinal}}\}$.  Then for any ordinal $\xi$, there exists a Banach space $X$ with $Sz(X)=\xi$ if and only if $\xi\in \Gamma\setminus \Lambda$.  

\label{maintheorem4}
\end{theorem}

This answers a question of Brooker \cite{BrookerAsplund}, who showed that the class of ordinals $\xi$ for which there exists an operator with Szlenk index $\omega^\xi$ is precisely the class of ordinals with countable cofinality.  The only difference between the space and operator version is that those ordinals of the form $\omega^{\omega^\xi}$ where $\xi$ is a limit ordinals may arise as the Szlenk index of an operator, but not as the Szlenk index of a Banach space.  This is due to the fact that for a Banach space, the $\ee$-Szlenk indices are submultiplicative, while for an operator, this is not necessarily true.

\section{Proofs of Theorems \ref{maintheorem}, \ref{maintheorem3}, \ref{maintheorem2}}

Recall that a real Banach space $X$ is said to be \emph{Asplund} if every real-valued, convex, continuous function defined on a convex, open subset $U$ of $X$ is Fr\'{e}chet differentiable on a dense $G_\delta$ subset of $U$.   If $X$ is a complex Banach space, we say it is Asplund if it is Asplund as a vector space over the reals.  We say that a subset $K\subset X^*$ is $w^*$-\emph{fragmentable} if for every non-empty subset $L$ of $K$ and any $\ee>0$,  there exists a $w^*$-open set $V$ in $X^*$ such that $\text{diam}(V\cap L) \leqslant \ee$ and $V\cap L\neq \varnothing$. It is easy to see that $Sz(K)=\infty$ if and only if $K$ fails to be $w^*$-fragmentable.   

We first collect some facts concerning the Szlenk index.  

\begin{proposition}Let $X$ be a Banach space and let $K\subset X^*$ be $w^*$-compact.  \begin{enumerate}[(i)]\item $Sz(K)=1$ if and only if $K$ is norm compact.  In particular, $Sz(X)=1$ if and only if $\dim X<\infty$.  \item $Sz(X)<\infty$ if and only if $X$ is Asplund. \item If $K\subset L$ for some $L\subset X^*$ $w^*$-compact, $Sz(K)\leqslant Sz(L)$.  \item If $Y$ embeds isomorphically into $X$, $Sz(Y)\leqslant Sz(X)$.  \item If $K$ is convex and non-empty and $Sz(K)<\infty$, then there exists an ordinal $\xi$ such that $Sz(K)=\omega^\xi$.  That is, $Sz(K)$ is a gamma number.  \item For $\ee>0$, $x^*\in s_\ee(K)$ if and only if there exist a directed set $D$ and nets $(x^*_\lambda)_{\lambda\in D}$, $(y^*_\lambda)_{\lambda\in D}\subset K$ each of which converges $w^*$ to $x^*$ such that for each $\lambda\in D$, $\|x^*_\lambda-y^*_\lambda\|>\ee$.  \item For $\ee>0$, if $x^*\in s_\ee(K)$, there exists a net $(x^*_\lambda)\subset K$ converging $w^*$ to $x^*$ such that $\|x^*_\lambda-x^*\|>\ee/2$ for every $\lambda$.    \item If there exists a net $(x^*_\lambda)\subset K$ converging $w^*$ to $x^*$ such that $\|x^*_\lambda-x^*\|>\ee$ for every $\lambda$, then $x^*\in s_\ee(K)$.  \item For any $w^*$-compact sets $K_1, \ldots, K_n$ and any $\ee>0$, $s_\ee(\cup_{i=1}^n K_i)\subset \cup_{i=1}^n s_{\ee/2}(K_i)$.  \item If $L\subset X^*$ is $w^*$-compact and $K,L$ are convex and non-empty, $$Sz(K+L)=\max\{Sz(K), Sz(L)\}.$$ \item If $A:X\to Y$ and $B:X_0\to Y_0$ are such that $A$ factors through $B$, then $Sz(A)\leqslant Sz(B)$. \item If $j:X\to Y$ is an isomorphic embedding, $Sz(j)=Sz(X)$. \item For any operator $A:X\to Y$ and any $Z\leqslant X$, $Sz(A|_Z)\leqslant Sz(A)$.  \end{enumerate}

\label{szlenk facts}
\end{proposition}

Items $(iii)$, $(vi)$, $(vii)$, and $(viii)$ follow easily from the definition.  Item $(i)$ can be found in \cite{BrookerAsplund}.  Item $(ii)$ follows from the fact that a Banach space is an Asplund space if and only if its dual ball is $w^*$-fragmentable, which is shown in Chapter I.$5$ of \cite{DGZ} for real Banach spaces and can be deduced for the complex case using the real case (see \cite[Theorem $1.4$]{BrookerAsplund} for the details).    Item $(iv)$ is given in \cite{HLM}.      Item $(ix)$ is given in \cite{BrookerDirect}.   Last, item $(x)$ was shown in \cite{Causey}.    Items $(xi)$, $(xii)$, and $(xiii)$ were shown by Brooker \cite{BrookerAsplund}. Item $(v)$ is given in \cite{Lancien} when $K=B_{X^*}$, but the same proof works for convex sets. For the ease of the reader, we sketch the proof here.  

\begin{proof}[Proof of $(v)$] First we claim that for any $K,L\subset X^*$ $w^*$-compact,  any $\ee>0$, and any ordinal $\xi$, $s_\ee^\xi(K)+L\subset s_\ee^\xi(K+L).$  The case $\xi=0$ is trivial. Next, assume $x^*+y^*\in s_\ee^{\xi+1}(K)+L$ with $x^*\in s_\ee^{\xi+1}(K)$, $y^*\in L$, and assume $s^\xi_\ee(K)+L\subset s_\ee^\xi(K+L)$.    Since $x^*\in s_\ee^{\xi+1}(K)$, for some directed set $D$, we may fix two nets $(x^{1,*}_\lambda)_{\lambda\in D}$, $(x^{2,*}_\lambda)_{\lambda\in D}$ contained in $s_\ee^\xi(K)$, converging $w^*$ to $x^*$, such that $\|x^{1,*}_\lambda- x^{2,*}_\lambda\|>\ee$ for all $\lambda\in D$.  This is by $(vi)$ above.  Also by $(vi)$, the nets $(x^{1,*}_\lambda+y^*)_{\lambda\in D}$, $(x^{2,*}_\lambda+y^*)_{\lambda\in D}$ witness that $x^*+y^*\in s_\ee(s_\ee^\xi(K)+L)\subset s_\ee(s_\ee^\xi(K+L))=s_\ee^{\xi+1}(K+L)$.   This proves the successor case. Last, assume $\xi$ is a limit ordinal and $s_\ee^\zeta(K)+L\subset s_\ee^\zeta(K+L)$ for all $\zeta<\xi$.   Then $$s_\ee^\xi(K)+L=\bigl(\cap_{\zeta<\xi}s_\ee^\zeta(K)\bigr)+L\subset \cap_{\zeta<\xi}(s_\ee^\zeta(K)+L)= s_\ee^\xi(K+L).$$  This finishes the claim.   

Now an easy homogeneity argument yields that for any $K\subset X^*$ $w^*$-compact and any $c,\ee>0$, $c s_\ee(K)=s_{c\ee}(cK)$. By induction, $c s_\ee^\xi(K)=s_{c\ee}^\xi(cK)$ for any ordinal $\xi$.  Finally, if $K$ is convex and $\ee>0$, $\xi\in \textbf{Ord}$, for any $n\in \nn$ and $0\leqslant j<n$, repeated applications of the claim from the beginning of the proof yield that \begin{align*} n s_{\ee/n}^{\xi n}(K) & = s_\ee^{\xi n}(nK) \supset s_\ee^{\xi(n-1)}\Bigl(K+\ldots +K+s^\xi_\ee(K)\Bigr) \\ & \supset s_\ee^{\xi(n-2)}\Bigl(K+\ldots +K +s^\xi_\ee(K)+s_\ee^\xi(K)\Bigr)\\ & \supset \ldots \supset s^\xi_\ee(K)+\ldots +s^\xi_\ee(K).\end{align*} From this we see that if $Sz_\ee(K)>\xi$, $Sz_{\ee/n}(K)>\xi n$.  Suppose $K\subset X^*$ is $w^*$-compact, convex, non-empty, and $Sz(K)<\infty$. Fix $\xi<Sz(K)$ and $\ee>0$ such that $s^\xi_\ee(K)\neq \varnothing$.  Then $s^{\xi\cdot  2}_{\ee/2}(K)\neq \varnothing$, and $\xi\cdot 2<Sz(K)$.  This means that $Sz(K)$ is a non-zero gamma number.  As mentioned above, the non-zero ordinals which are gamma numbers are precisely the ordinals of the form $\omega^\gamma$, $\gamma\in \textbf{Ord}$.

\end{proof}

We state the lemmas which, in conjunction with Theorem \ref{theorem}, yield Theorems \ref{maintheorem} and \ref{maintheorem2}.  We postpone the proof of Theorem \ref{theorem} until Section $4$.  In what follows, $\mathbb{F}$ will denote the scalar field (either real or complex numbers) and for $K\subset X^*$ and $T\subset \mathbb{F}$,  $$T K:=\{t x^*: t\in T,  x^*\in K\}.$$ We let $S_\mathbb{F}=\{t\in \mathbb{F}: |t|=1\}$.

\begin{lemma} For any $w^*$-compact subset $K$ of $X^*$, any $\ee_0>\ee>0$, and any ordinal $\xi$, $s^\xi_{\ee_0}(S_\mathbb{F}K)\subset S_\mathbb{F} s_{\ee/2}^\xi(K)$. In particular, $Sz_\ee(S_\mathbb{F}K)\leqslant     Sz_{\ee/3}(K)$ and $Sz(S_\mathbb{F}K)=Sz(K)$.  

\label{lemma1}
\end{lemma}

\begin{proof} We will use Proposition \ref{szlenk facts}$(vii)$, $(viii)$ frequently throughout the proof.   We prove the result by induction.  The base case is trivial.  To see the successor case, it is sufficient to treat the $\xi=1$ case.  Assume $y^*\in s_{\ee_0}(S_\mathbb{F}K)$.  Fix $(\ee_\lambda x^*_\lambda)\subset S_\mathbb{F}K$ converging $w^*$ to $y^*$ such that for all $\lambda$, $\ee_\lambda\in S_{\mathbb{F}}$, $x^*_\lambda\in K$, and $\|\ee_\lambda x^*_\lambda - y^*\|>\ee_0/2$.   By passing to a subnet, we may assume $\ee_\lambda\to \ee\in S_\mathbb{F}$ and $x^*_\lambda\underset{w^*}{\to}x^*\in K$. Note that $y^*=\ee x^*$. Then $$\underset{\lambda}{\lim\sup} \|x^*_\lambda - x^*\| = \underset{\lambda}{\lim\sup} \|\ee_\lambda x^*_\lambda - \ee x^*\| \geqslant \ee_0/2>\ee/2.$$  From this it follows that $x^*\in s_{\ee/2}(K)$ and $y^*=\ee x^*\in S_\mathbb{F}s_{\ee/2}(K)$.   

Last, assume that the result holds for every $\zeta<\xi$, where $\xi$ is a limit ordinal.  Assume $y^*\in s^\xi_{\ee_0}(S_\mathbb{F}K)$.  For every $\zeta<\xi$, there exist $\ee_\zeta\in S_\mathbb{F}$ and $x^*_\zeta\in s^\zeta_{\ee/2}(K)$ such that $y^*=\ee_\zeta x^*_\zeta$.  By passing to subnets of $(\ee_\zeta)_{\zeta<\xi}$ and $(x^*_\zeta)_{\zeta<\xi}$, we may assume $\ee_\zeta\to \ee\in S_\mathbb{F}$ and $x^*_\zeta\underset{w^*}{\to}x^*$.  Then because the sets $(s^\zeta_{\ee/2}(K))_{\zeta<\xi}$ are decreasing and $w^*$-closed, $x^*\in s_{\ee/2}^\xi(K)$ and $y^*=\ee x^*\in S_\mathbb{F}s_{\ee/2}^\xi(K)$.

\end{proof}

\begin{lemma} Let $K$ be a compact, Hausdorff topological space. Let $D_K=\{\delta_\varpi:\varpi\in K\}\subset B_{C(K)^*}$, where $\delta_\varpi$ denotes the Dirac measure at $\varpi$.    Then for any $0<\ee<2$, $Sz_\ee(D_K)=i(L)$.

\label{lemma2}
\end{lemma}

\begin{proof} Fix $0<\ee<2$. For any closed subset $L$ of $K$, let $D_L=\{\delta_\varpi: \varpi\in L\}\subset B_{C(K)^*}$.  We show by induction on $\xi$ that $s_\ee^\xi(D_L)= D_{L^\xi}$.  The base and limit ordinal cases are trivial.   In order to complete the successor case, it is sufficient to check the case that $\xi=1$.    Note that $\delta:L\to D_L$ given by $\delta(\varpi)=\delta_\varpi$ is a homeomorphism when $D_L$ is endowed with its $w^*$-topology.  Therefore if $\delta_\varpi\in D_{L'}$, $\varpi\in L'$, and $\varpi$ is not isolated in $L$.  This means $\delta_\varpi$ is not $w^*$-isolated in $\delta(L)=D_L$, and any $w^*$-neighborhood of $\delta_\varpi$ must contain two distinct members of $D_L$, and therefore have diameter exactly $2>\ee$.  From this it follows that $\delta_\varpi\in s_\ee(L)$.   Conversely, if $\delta_\varpi\in s_\ee(D_L)$, then $\varpi$ cannot be isolated in $L$, otherwise $\{\delta_\varpi\}$ is a $w^*$-neighborhood of $\delta_\varpi$ relative to $D_L$, and $\text{diam}(\{\delta_\varpi\}\cap D_L)=0<\ee$.   Thus $\varpi\in L'$, and $\delta_\varpi\in D_{L'}$.

\end{proof}

Here we isolate the geometric version of the Hahn-Banach theorem which allows us to apply Theorem \ref{theorem} to obtain Theorems \ref{maintheorem}, \ref{maintheorem3}, and \ref{maintheorem4}.  We discuss the proof for the sake of completeness.  

\begin{theorem}[Hahn-Banach] Let $Z$ be a locally convex topological vector space over $\mathbb{F}$ (either the real or complex numbers). If $C,D$ are disjoint, convex, compact subsets of $Z$, there exists a continuous, linear functional $\phi: Z\to \mathbb{F}$ and real numbers $\alpha<\beta$ such that $\max_{z\in C}\text{\emph{Re}\ }\phi(z)\leqslant \alpha<\beta \leqslant \min_{z\in D}\text{\emph{Re\ }}\phi(z)$.

\label{hb}
\end{theorem}

Any continuous functional $\phi$ on $Z$ such that there exist real numbers $\alpha<\beta$ as in the conclusion of the theorem are said to \emph{separate} $C$ and $D$. With this, we deduce that if $X$ is a Banach space, $K\subset L\subset X^*$ are $w^*$-compact, $L$ is convex, and $\max_{x^*\in L}\text{Re\ }x^*(x)=\max_{x^*\in K}\text{Re\ }x^*(x)$ for all $x\in X$, then $L=\overline{\text{co}}^{w^*}(K)$. To see this, we fix $x^*_0\in L$ and apply Theorem \ref{hb} with $Z=X^*$ endowed with its $w^*$-topology, $C=\overline{\text{co}}^{w^*}(K)$, $D=\{x^*_0\}$. We recall that a continuous functional $\phi:(X^*, w^*)\to \mathbb{F}$ is of the form $\phi(x^*)=x^*(x)$ for some $x\in X$.  If $x^*_0\notin C$, there would exist $x\in X$ such that $$\max_{x^*\in K}\text{Re\ }x^*(x)<\text{Re\ }x^*_0(x),$$ contradicting our hypotheses.  Therefore $x^*_0\in \overline{\text{co}}^{w^*}(K)$, and $L\subset \overline{\text{co}}^{w^*}(K)$.  The reverse inclusion is clear, and $L=\overline{\text{co}}^{w^*}(K)$.

\begin{proof}[Proof of Theorem \ref{maintheorem}] Note that if $D_K$ is the collection of Dirac measures, then by Theorem \ref{hb},  $B_{C(K)^*}=\overline{\text{co}}^{w^*}(S_\mathbb{F}D_K)$. This is also a consequence of the Krein-Milman theorem, since the unimodular multiples of Dirac measures are precisely the extreme points of $B_{C(K)^*}$. By Theorem \ref{theorem} and Lemmas \ref{lemma1} and \ref{lemma2}, $$Sz(C(K))=Sz(B_{C(K)^*})= \Gamma(Sz(S_\mathbb{F}D_K))=\Gamma(Sz(D_K))=\Gamma(i(K)).$$

\end{proof}

Samuel \cite{Samuel} showed that for a countable ordinal $\xi$, the Szlenk index of $C([1, \omega^{\omega^\xi}])$ is $\omega^{\xi+1}$.  H\'{a}jek and Lancien \cite{HL} deduced that the Szlenk index of $C([0,\xi])$ is $\omega_1\omega$ when $\omega_1\leqslant \xi < \omega_1\omega$. Brooker \cite{Brooker} gave a computation of the Szlenk index of $C([0,\xi])$ for every infinite ordinal $\xi$.  Brooker showed that for any infinite ordinal  $\xi$, the Szlenk index of $C([0,\xi])$ is $\omega^{\zeta+1}$, where $\zeta$ is such that $\omega^{\omega^\zeta}\leqslant \xi<\omega^{\omega^{\zeta+1}}$.  It is easy to show that for any ordinals $\zeta, \xi$ such that $\omega^\zeta \leqslant \xi <\omega^{\zeta+1}$, $i([0, \xi])=\zeta+1$.  From this, it follows that for any ordinals $\omega^{\omega^\zeta}\leqslant \xi<\omega^{\omega^{\zeta+1}}$, $\Gamma(i([0, \xi]))=\omega^{\zeta+1}$, so Theorem \ref{maintheorem} offers another computation of the Szlenk indices of all $C([0, \xi])$ spaces.  Theorem \ref{maintheorem} also offers another proof that $C(K)$ is an Asplund space if and only if $K$ is scattered, which was originally shown by Namioka and Phelps \cite{NP}.

In the next lemma, we will say that for a Banach space $F$ and a subset $C\subset F^*$ is $1$-\emph{norming} for $F$ provided that for every $x\in F$, $\sup_{x^*\in C}\text{Re\ }x^*(x) \geqslant \|x\|$.    

\begin{lemma}\begin{enumerate}[(i)]\item If $E,F$ are Banach spaces, $B\subset E^*$, $C\subset F^*$ are $w^*$-compact, and $f:B\to C$ is a $w^*$-$w^*$-continuous, Lipschitz surjection, then $Sz(C)\leqslant Sz(B)$.  \item If $E$ is a Banach space, $K,L\subset E^*$ are $w^*$-compact and non-empty, $F$ is a Banach space,  $C\subset F^*$ is $w^*$ compact, and $f:K+L\to C$ is $w^*$-$w^*$-continuous, Lipschitz, and surjective, then $Sz(\overline{\text{\emph{co}}}^{w^*}(C))\leqslant \max\{\Gamma(Sz(K)), \Gamma(Sz(L))\}$.    \end{enumerate}

\label{just}
\end{lemma}

\begin{proof}$(i)$ The proof is the same as the proof Lemma $2.5$ of \cite{BrookerAsplund}, where it was shown in the case that $f$ is the restriction to $B$ of $S^*$, where $S:F\to E$ is a linear operator.  The proof is easily seen to hold in the more general setting of Lemma \ref{just}.

$(ii)$ If $\max\{\Gamma(Sz(K),\Gamma(Sz(L))\}=\infty$, there is nothing to show, so assume both $K$ and $L$ are $w^*$-fragmentable. Then $\max\{\Gamma(Sz(K)), \Gamma(Sz(L))\}=\omega^\xi$ for some ordinal $\xi$.   By $(i)$, Proposition \ref{szlenk facts}$(x)$, and Theorem \ref{theorem},  \begin{align*} Sz(C) & \leqslant Sz(K+L)\leqslant Sz(\overline{\text{co}}^{w^*}(K)+\overline{\text{co}}^{w^*}(L)) \\ & =\max\{Sz(\overline{\text{co}}^{w^*}(K), Sz(\overline{\text{co}}^{w^*}(L)\}=\max\{\Gamma(Sz(K)), \Gamma(Sz(L))\}=\omega^\xi.\end{align*} Another application of Theorem \ref{theorem} yields that $$Sz(\overline{\text{co}}^{w^*}(C))=\Gamma(Sz(C))\leqslant \omega^\xi.$$

\end{proof}

Let $L(Y^*,X)$ denote the bounded operators from $Y^*$ into $X$. Given $x\in X$, $y\in Y$, we let $x\otimes y:Y^*\to X$ be the operator given by $(x\otimes y)(y^*)=y^*(y)x$.  We recall that the injective tensor product $\tens$ of $X,Y$ can be viewed as the norm closure in $L(Y^*,X)$ of $X\otimes Y:=\{\sum_{i=1}^n x_i\otimes y_i:n\in \nn, x_i\in X, y_i\in Y\}$.   

Note that if $A:X\to X_1$, $B:Y\to Y_1$ are operators, one may define the operator $A\otimes B:X\hat{\otimes}_\ee Y\to X_1\hat{\otimes}_\ee Y_1$ to be the unique continuous extension of the operator from $X\otimes Y$ into $X_1\otimes Y_1$ given by $\sum_{i=1}^n x_i\otimes y_i\mapsto \sum_{i=1}^n Ax_i \otimes By_i$.  This is well-defined and linear with norm not exceeding $\|A\|\|B\|$.  To see this, note that if $u=\sum_{i=1}^n x_i\otimes y_i$ is viewed as a a member of $L(Y^*,X)$, $\sum_{i=1}^n Ax_i\otimes By_i$, as viewed as a member of $L(Y^*_1, X_1)$, is $AuB^*$.

\begin{proof}[Proof of Theorem \ref{maintheorem3}] Since $A,B$ factor through $A\otimes B$, it is clear that $$Sz(A\otimes B)\geqslant \max\{Sz(A), Sz(B)\}.$$   We focus on the reverse inequality.  Let $E=X\oplus_1 Y$, $M_1=A^*B_{X_1^*}\times \{0\}\subset E^*$, $M_2=B^*B_{Y_1^*}$, $f:M_1+M_2\to C:=f(M_1+M_2)\subset (X\hat{\otimes}_\ee Y)^*$ be given by $f(A^*x^*, B^*y^*)=A^*x^*\otimes B^*y^*$.  It is straightforward to check that $f$ is $w^*$-$w^*$ continuous, Lipschitz, and onto.  Using Theorem \ref{hb}, $\overline{\text{co}}^{w^*}(C)=(A\otimes B)^* B_{(X_1\hat{\otimes}_\ee Y_1)^*}$. Indeed, for any $u=\sum_{i=1}^n x_i\otimes y_i\in X\hat{\otimes}_\ee Y$, \begin{align*} \sup_{z^*\in (A\otimes B)^*B_{(X_1\hat{\otimes}_\ee Y_1)^*}} \text{Re\ }z^*(u) & = \|A\otimes B (u)\|=\max_{x^*\in B_{X_1^*}, y^*\in B_{Y_1^*}} \text{Re\ }\sum_{i=1}^n x^*(Ax_i)y^*(By_i) \\ &= \max_{(A^*x^*,B^*y^*)\in C} \text{Re\ }(A^*x^*\otimes B^*y^*)(u).  \end{align*}

By convexity of $M_1$ and $M_2$ and Lemma \ref{just}$(ii)$, $$Sz(A\otimes B) \leqslant \max\{Sz(M_1), Sz(M_2)\}=\max\{Sz(A), Sz(B)\}.$$

The final part follows by applying the first part to $I_X\otimes I_Y$.

\end{proof}

\section{Proof of Theorem \ref{maintheorem4}}

The purpose of this section is to provide a complete characterization of those ordinals which occur as the Szlenk index, the Bourgain $\ell_1$ index, or the Bourgain $c_0$ index of a Banach space.  We know that the Szlenk index of an Asplund space must be of the form $\omega^\xi$.   It turns out that the solution of this problem for the Szlenk index can be done in a way which simultaneously solves the problem of which ordinals occur as the $\ell_1$ or $c_0$ index of a Banach space.

If $\Lambda$ is a set, we let $\Lambda^{<\nn}$ denote the finite sequences in $\Lambda$, including the empty sequence $\varnothing$.    We order $\Lambda^{<\nn}$ by $s\preceq t$ if $s$ is an initial segment of $t$, and $s\prec t$ if $s$ is a proper initial segment of $t$.    We let $|t|$ denote the length of $t$. For $0\leqslant i\leqslant |t|$,  $t|_i$ denotes  the initial segment of $t$ having length $i$, and if $t\neq \varnothing$, $t^-$ denotes the maximal, proper initial segment of $t$.  We let $s\cat t$ denote the concatenation of $s$ with $t$. Given a subset $U$ of $\Lambda^{<\nn}$, we let $MAX(U)$ denote the members of $U$ which are maximal in $U$ with respect to $\prec$. If $T\subset \Lambda^{<\nn}$ is closed under taking initial segments, we say $T$ is a \emph{tree on } $\Lambda$.  We define $T'=T\setminus MAX(T)$.     We define the higher order derived trees $T^\xi$  by $T^0=T$, $T^{\xi+1}=(T^\xi)'$, and $T^\xi=\cap_{\zeta<\xi} T^\zeta$ when $\xi$ is a limit ordinal.  These are indeed trees if $T$ is. We define the order of $T$ by $o(T)=\min\{\xi: T^\xi=\varnothing\}$ if this class is non-empty, and $o(T)=\infty$ otherwise. We say $T$ is \emph{well-founded} if $o(T)<\infty$, and we say $T$ is \emph{ill-founded} otherwise.   A \emph{branch} of $T$ is a maximal linearly ordered subset of $T$.  We note that $o(T)=\infty$ if and only if $T$ admits an infinite branch if and only if $T$ admits an infinite chain.  We say $T\subset \Lambda^{<\nn}\setminus\{\varnothing\}$ is a $B$-\emph{tree} if $T\cup \{\varnothing\}$ is a tree.  Each of the notions above regarding derived trees and order can be relativized to $B$-trees.   If $\Lambda_1$, $\Lambda_2$ are two sets, we naturally identify $(\Lambda_1\times \Lambda_2)^{<\nn}$ with the subset of $\Lambda_1^{<\nn}\times \Lambda_2^{<\nn}$ consisting of those pairs of sequences $(s,t)$ such that $|s|=|t|$.   That is, we identify the sequence of pairs $((a_i, b_i))_{i=1}^n$ with the pair of sequences $((a_i)_{i=1}^n, (b_i)_{i=1}^n)$, and we identify $\varnothing$ with $(\varnothing, \varnothing)$.   Given subsets $S,T$ of $\Lambda_1^{<\nn}$, $\Lambda_2^{<\nn}$, respectively, we say $\theta:S\to T$ is \emph{monotone} provided that for any $s\prec s_1$, $s, s_1\in S$, $\theta(s)\prec \theta(s_1)$.      We recall that for any $B$-trees $S$ and $T$, $o(S)\leqslant o(T)$ if and only if there exists a monotone, length-preserving function $\theta:S\to T$, where we obey the convention that $\xi<\infty$ for any ordinal $\xi$.  Since we are unaware of a proof of this fact in the literature, we include one for completeness.

\begin{lemma} For $B$-trees $S,T$, there exists a monotone function $\theta:S\to T$ such that $|\theta(s)|=|s|$ for all $s\in S$ if and only if $o(S)\leqslant o(T)$.

\label{monotone lemma}
\end{lemma}

\begin{proof} If $T$ is ill-founded, then there exists an infinite branch $(x_n)_{n\in \nn}$.  Since branches are maximal by definition, $|x_n|=n$ for each $n\in \nn$.  For any tree $S$, we may then define $\theta(s)=x_{|s|}$.   This is monotone and preserves lengths.

We next prove by induction on $A=\{(\zeta, \xi)\in \ord^2: \xi\leqslant \zeta\}$ with its lexicographical ordering that if $S,T$ are $B$-trees with $o(S)=\xi$ and $o(T)=\zeta$, then there exists a monotone, length-preserving function $\theta:S\to T$.  In the proof, if $S$ is a $B$-tree on the set $\Lambda$ and $s$ is a finite sequence in $\Lambda$, let $S(s)$ denote those non-empty members $u$ of $\Lambda^{<\nn}$ such that $s\cat u\in S$. Note that $S(s)$ is a $B$-tree, well-founded if $S$ is.     Suppose that $(\zeta, \xi)\in A$ and the result holds for every pair less than $(\zeta, \xi)$.   Let $S,T$ be $B$-trees (say on the set $\Lambda$) such that $o(S)=\xi$ and $o(T)=\zeta$.   Let $R$ denote the sequences in $S$ having length $1$.    If $R=\varnothing$, $S=\varnothing$, and the empty map from $S$ to $T$ is monotone and preserves lengths.   Assume that $R\neq \varnothing$.    Note that for any $s\in S$, $o(S(s))<\xi$.  Indeed, for any ordinal $\eta$, $S^\eta(s)=(S(s))^\eta$, which is easily seen by induction.  If $o(S(s))=0$, then since $s\in S$, $o(S)>0$.   If $o(S(s))=\eta+1$, $S^\eta(s)\neq \varnothing$, whence $s\in S^{\eta+1}$ and $o(S)>\eta+1$.  If $o(S(s))=\eta$ and $\eta$ is a limit ordinal, then for any $\gamma<\eta$, $S^\gamma(s)\neq \varnothing$, whence $s\in S^{\gamma+1}$.  Since this holds for every $\gamma<\eta$, $s\in \cap_{\gamma<\eta}S^\gamma=S^\eta$, and $o(S)>\eta$.    

We claim that for any $s\in R$, there exists a sequence  $t\in T$ having length $1$ such that $o(S(s))\leqslant o(T(t))$.    If it were true that for some $s\in S$, $o(S(s))>o(T(t))$ for every $t\in T$, then $o(S)>o(T)$.  Indeed, since $T^{o(T(t))}(t)=\varnothing$, $t$ has no proper extensions in $T^{o(T(t))}$, and $t\notin T^{o(T(t))+1}$.  Then if $\eta=(\sup_{t\in T}o(T(t)))+1$, $T^\eta=\varnothing$, since $T^\eta$ contains no sequences of length $1$, and must therefore be empty.  From this it follows that $\zeta\leqslant \eta$.       But $o(S(s))>\sup_{t\in T}o(T(t))$ implies that $\xi=o(S)>o(S(s))\geqslant \sup_{t\in T} o(T(t))+1=\eta\geqslant \zeta$, a contradiction.   This yields the claim.  

For every $s\in R$, let $t_s$ be a sequence in $T$ of length $1$ such that $o(S(s))\leqslant o((T(t_s))$.  By the inductive hypothesis, there exists $\theta_s:S(s)\to T(t_s)$ which is monotone and preserves lengths.  We then define $\theta:S\to T$ by letting $\theta(s)=t_s$ and if $s'$ is a sequence in $S$, $\theta(s')=t_s\cat \theta_s(s'')$, where $s\in R$ and $s''$ are the unique sequences such that $s$ has length $1$ and $s'=s\cat s''$.    The function $\theta$ defined in this way is monotone and preserves lengths.

\end{proof}

Given a well-founded tree $T$ and $t\in T$, we let $T(t)=\{s\in T: t\preceq s\}$. Sets of this form are called \emph{wedges}.  Then the \emph{coarse wedge topology} on a non-empty tree, with which all trees in this section will be endowed, is the topology for which the  wedges and their complements form a subbase. Clearly this is a Hausdorff topology on any tree. Indeed, if $s\neq t$ are two members of a tree, then either neither is an initial segment of the other, or we may assume $s\prec t$.  In this case, $t\in T(t)$, $s\in T\setminus T(t)$ witnesses that $s$ and $t$ may be separated by open sets.  If $T$ is well-founded, then $T$ is compact with its coarse wedge topology.  This was stated for a more general notion of tree in \cite{Nyikos}, and there the terminology ``branch complete'' was used in place of ``well-founded.''  Since our definition of tree is more restrictive than the one in \cite{Nyikos}, and this greater restriction makes for a simpler special case of the more involved proof in \cite{Nyikos},  we include the proof.   

\begin{proposition} If $T$ be a non-empty, well-founded tree endowed with its coarse wedge topology, then $T$ is compact.  

\end{proposition}

\begin{proof}   We will use the Alexander subbase theorem, which states that in order to see that $T$ is compact with its coarse wedge topology, it is sufficient to prove that any cover of $T$ by members of a subbase admits a finite subcover.  We will use the subbase for the coarse wedge topology consisting of wedges and their complements.  To that end, fix a cover of $T$ $$\mathcal{U}=\{T(t): t\in S_1\}\cup \{T\setminus T(t): t\in S_2\}.$$  If $S_2=\varnothing$, there exists $t\in S_1$ such that $\varnothing\in T(t)$.  Of course this means $t=\varnothing$, and $T=T(\varnothing)$.  Thus $\{T(\varnothing)\}\subset\mathcal{U}$ is a subcover with cardinality $1$. Thus we may assume $S_2\neq \varnothing$.  If $S_2$ contains two incomparable members $t_1, t_2$, then $T=(T\setminus T(t_1))\cup (T\setminus T(t_2))$, and $\{T(t_1), T(t_2)\}\subset \mathcal{U}$ is a subcover of cardinality $2$.  Thus we may assume $S_2$ is a non-empty chain in $T$.   Since $T$ is well-founded, every chain in $T$ is finite. Since $S_2$ is finite, it has a maximum member, say $t_2$.  Since $t_2\notin \cup_{t\in S_2}T\setminus T(t)$, there exists $t_1\in S_1$ such that $t_2\in T(t_1)$.  Then $T=T(t_1)\cup (T\setminus T(t_2))$, and $\{T(t_1), T\setminus T(t_2)\}\subset \mathcal{U}$ is a subcover of cardinality $2$.

\end{proof}

Note that $t$ lies in the Cantor-Bendixson derivative of $T$ if and only if $\{s\in T: s^-=t\}$ is infinite.   For any derived tree $T^\xi$, the relative coarse wedge topology on $T^\xi$ as a subset of $T$ coincides with the coarse wedge topology of $T^\xi$.    Let us say that the tree $T$ is \emph{full} if for every ordinal $\xi$ and every $t\in T^{\xi+1}$, $\{s\in T^\xi: s^-=t\}$ is infinite.  It follows from this condition that the $\xi^{th}$ derived tree coincides with the $\xi^{th}$ Cantor-Bendixson derivative of $T$ for every $\xi$.

In this section, a $1$-\emph{unconditional basis for} the  Banach space $X$ will be an unordered collection $(e_i)_{i\in I}$ of vectors in $X$ having dense span in $X$ such that for every finite subset $A$ of $I$, every set of unimodular scalars $(\ee_i)_{i\in A}$, and all scalars $(a_i)_{i\in A}$, $$\|\sum_{i\in A} \ee_ia_ie_i\|=\|\sum_{i\in A}a_ie_i\|.$$  Note that in this case, for every non-empty subset $A$ of $I$, the projection $P_A:X\to X$ given by $P_A\sum_{i\in I}a_ie_i=\sum_{i\in A} a_ie_i$ is a projection of norm $1$. If $x^*, y^*\in X^*$ are such that there exists a subset $A$ of $I$ such that $x^*=P^*_A y^*$, we will say that $x^*$ is a \emph{suppression projection} of $y^*$, and in this case, $\|x^*\|\leqslant \|y^*\|$.    For any $\ee=(\ee_i)_{i\in I}\subset S_\mathbb{F}^I$, the operator $U_\ee:X\to X$ given by $U_\ee\sum_{i\in I} a_ie_i = \sum_{i\in I}\ee_ia_ie_i$ is an isometry.  We note that if $(x^*_\lambda)\subset X^*$ is a bounded net converging $w^*$ to $x^*$ and $\ee_\lambda=(\ee_{\lambda, i})_{i\in I}$ is a net in $S_\mathbb{F}^I$ converging to $\ee$ in the product topology, $U_{\ee_\lambda}^*x^*_\lambda\underset{w^*}{\to} U_\ee^* x^*$.  Let $(e_i^*)_{i\in I}$ denote the biorthogonal functionals to $(e_i)_{i\in I}$.  Note that $X^*$ is identifiable with all formal (not necessarily norm converging) series $\sum_{i\in I} a_ie_i^*$ such that $\sup_{|A|<\infty} \|\sum_{i\in A}a_ie_i^*\|<\infty$. For $x^*\in X^*$, we let $\supp(x^*)=\{i\in I: x^*(e_i)\neq 0\}$. Note that $X^*$ can be endowed with a lattice structure where the positive elements are those formal series $\sum_{i\in I}a_ie_i^*\in X^*$ with $a_i\geqslant 0$ for all $i$.  We may define $|\cdot|$ on $X^*$ by $|\sum_{i\in I} a_ie_i^*|=\sum_{i\in I}|a_i|e_i^*$, and note that this function is $w^*$-$w^*$-continuous on $B_{X^*}$. Given a subset $M$ of $X^*$, we let $|M|=\{|f|: f\in M\}$. Note that if $M$ is $w^*$-compact, so is $|M|$. We say $M\subset X^*$ is \emph{unconditional} if $M=\{U_\ee^* x^*: x^*\in M, \ee\in S_\mathbb{F}^I\}$.    The \emph{unconditionalization} of $M$ is $\{U^*_\ee x^*: \ee\in S_\mathbb{F}^I, x^*\in M\}$.  Our remarks imply that the unconditionalization of $M$ is $w^*$-compact if $M$ is.  

\begin{lemma} For any $M\subset [e_i^*:i\in I]$ which is $w^*$-compact, any ordinal $\zeta$,  any $0<\ee<\ee_1$, $$|s_{\ee_1}^\zeta(M)|\subset s^\zeta_{\ee/2}(|M|).$$ In particular, $Sz(M)= Sz(|M|)$.  

Here, $[e_i^*:i\in i]$ denotes the norm closure in $X^*$ of the span of $\{e^*_i: i\in I\}$. 
  
\label{unconditionalization}
\end{lemma}

\begin{proof} If $M_1$ is the unconditionalization of $M$, $M\subset M_1$ and $|M|=|M_1|$. If we can prove that for any ordinal $\zeta$, any $0<\ee<\ee_1$, $$|s_{\ee_1}^\zeta(M_1)|\subset s^\zeta_{\ee/2}(|M_1|),$$ then $$|s_{\ee_1}^\zeta(M)|\subset |s_{\ee_1}^\zeta(M_1)|\subset s^\zeta_{\ee/2}(|M_1|)=s^\zeta_{\ee/2}(|M|). $$  Thus we may replace $M$ by its unconditionalization and assume $M$ itself is unconditional.

 Since $M$ is unconditional,  $|M|\subset M$ and $Sz(|M|)\leqslant Sz(M)$.    We will prove the first statement of the lemma by induction.  The base and limit cases are trivial.   Suppose the result holds for some $\zeta$ and suppose also that $x^*\in s_{\ee_1}^{\zeta+1}(M)$.  We may fix a net $(x^*_\lambda)\subset s^\zeta_{\ee_1}(M)$ converging $w^*$ to $x^*$ such that for all $\lambda$, $\|x^*_\lambda-x^*\|>\ee_1/2$.  For every $\lambda$, fix $\ee_\lambda=(\ee^\lambda_i)_{i\in I}\subset S_\mathbb{F}^I$ such that $U_{\ee_\lambda}^* x^*_\lambda=|x^*_\lambda|\in s_{\ee/2}^\zeta(|M|)$. By the inductive hypothesis, $U_{\ee_\lambda}^* x^*_\lambda\in s_{\ee/2}^\zeta(|M|)$.   By passing to a subnet, we may assume there exists $\ee\in S_\mathbb{F}^I$ such that $\ee_\lambda\to \ee$ in the product topology and note that necessarily $U_\ee^* x^*=|x^*|$.   Moreover, since $x^*$ lies in the norm closure of the span of $(e_i^*:i\in I)$, it follows that $U_{\ee_\lambda}^* x^*\to |x^*|$ in norm.   From this it follows that \begin{align*} \underset{\lambda}{\lim\sup} \||x^*_\lambda|-|x^*|\| &  \geqslant \underset{\lambda}{\lim\sup}\|U_{\ee_\lambda}^*x^*_\lambda- U^*_{\ee_\lambda} x^*\| - \|U^*_{\ee_\lambda} x^*- |x^*|\| \\ & = \underset{\lambda}{\lim\sup}\|x^*_\lambda-x^*\|\geqslant \ee_1/2>\ee/2.\end{align*}    This yields that $|x^*|\in s^{\zeta+1}_{\ee/2}(|M|)$.

\end{proof}

We include here another easy fact.

\begin{claim} Suppose that $X$ is a Banach space and $(M_i)_{i\in I}$ is a collection of $w^*$-compact subsets of $X^*$ such that for every $i \neq j\in I$, $\{0\}=M_i\cap M_j$.    Let $M=\underset{i\in I}{\bigcup}M_i$. Suppose that any net $(x^*_\lambda)\subset M$ such that for any $i\in I$, $(x^*_\lambda)$ is eventually in $M\setminus (M_i\setminus\{0\})$ converges $w^*$ to $0$.    \begin{enumerate}[(i)]\item   $M$ is $w^*$-compact.  \item For any ordinal $\mu$ and any $\ee>0$, $s_\ee^\mu(M)\setminus\{0\} \subset \underset{i\in I}{\bigcup} s_\ee^\mu(M_i)$. Consequently, for any $\ee>0$, if $\sup_{i\in I}Sz_\ee(M_i)<\infty$, then $Sz_\ee(M)\leqslant (\sup_{i\in I}Sz_\ee(M_i))+1$.  Moreover, if for some $\xi$, some $\ee>0$, and every $i\in I$, $s^\xi_\ee(M_i)\subset \{0\}$, $Sz_\ee(M)\leqslant \xi+1$.

\end{enumerate}

\label{trivial claim}
\end{claim}

\begin{proof}$(i)$ If $I$ is finite, $(i)$ is trivial.  Assume $I$ is infinite and let $I_\infty=I\cup \{\infty\}$ denote the one-point compactification of $I$ with its discrete topology.  For any net $(x^*_\lambda)\subset M$, either $(x^*_\lambda)$ has a subnet which is constantly $0$, and which therefore converges $w^*$ to $0$, or $(x^*_\lambda)$ has a subnet (which we assume is the entire net) such that for all $\lambda$, $x^*_\lambda \neq 0$.  In this case, for each $\lambda$, there exists a unique  $i_\lambda\in I$ such that $x^*_\lambda\in M_{i_\lambda}$.   Then we may pass to a subnet and assume that $(i_\lambda)$ converges in $I_\infty$.   If $i_\lambda$ converges to $i\in I$, we may pass to a subnet once more and assume $(x^*_\lambda)\subset M_i$.  By $w^*$-compactness of $M_i$, $(x^*_\lambda)$ has a $w^*$-convergent subnet.  Otherwise $i_\lambda$ converges to $\infty$ in $I_\infty$, whence for any $i\in I$, $(x^*_\lambda)$ is eventually not in $M_i$.    From the hypothesis, $x^*_\lambda\underset{w^*}{\to}0$. Thus in any case, the original net $(x^*_\lambda)\subset M$ has a $w^*$-converging subnet.

$(ii)$ It is well-known and an easy proof by induction on $\mu$ that if $N,P$ are $w^*$-compact subsets of $X^*$ and $V$ is a $w^*$-open set such that $V\cap P=\varnothing$, then for any $\ee>0$ and $\mu>0$,   $s_\ee^\mu(N\cup P)\cap V\subset s_\ee^\mu(N)$.   Then for any $i\in I$, we may let $N=M_i$, $P=\cup_{i\neq j\in I}M_j$, and $V=X^*\setminus P$.  That $P$ is $w^*$-compact follows from applying $(i)$ to the collection $(M_j)_{j\in I\setminus\{i\}}$.   We deduce that $$s_\ee^\mu(N\cup P)\cap V \subset s^\mu_\ee(M_i)$$ for any $i\in I$.  If $f\in s_\ee^\mu(M)\setminus\{0\}$, then $f\in M_i\setminus \{0\}$ for a unique $i\in I$, whence $f\in s_\ee^\mu(M)\cap (X^*\setminus \underset{i\neq j\in I}{\bigcup} M_j)\subset s_\ee^\mu(M_i)$.  Thus it follows that $s^\mu_\ee(M)\setminus\{0\}\subset \cup_{i\in I}s_\ee^\mu(M_i)$.   This means that if $\xi=\sup_{i\in I} Sz_\ee(M_i)<\infty$, $s^\xi_\ee(M)\subset \{0\}$, and $Sz_\ee(M)\leqslant \xi+1$.    If $\xi\in \ord$ and $\ee>0$ are such that for every $i\in I$, $s_\ee^\xi(M_i)\subset \{0\}$, $s_\ee^\xi(M)\subset \cup_{i\in I}s_\ee^\xi(M_i)\subset \{0\}$, whence $Sz_\ee(M)\leqslant \xi+1$.

\end{proof}

We next define the Bourgain $\ell_1$ and $c_0$ indices, which originated in \cite{Bo}.  Given a Banach space $X$ and $K\geqslant 1$, we let $T_1(X,K)$ denote the tree in $X$ consisting of the empty sequence together with those sequences $(x_i)_{i=1}^n\subset X$ such that for all scalars $(a_i)_{i=1}^n$, $$K^{-1}\sum_{i=1}^n |a_i|\leqslant \|\sum_{i=1}^n a_ix_i\|\leqslant \sum_{i=1}^n |a_i|.$$  We let $I_1(X)=\sup_{K\geqslant 1}o(T_1(X,K))$.  We note that $I_1(X)=\infty$ if and only if $\ell_1$ embeds into $X$.  Similarly, we define $T_\infty(X,K)$ to be the empty sequence together with those sequences $(x_i)_{i=1}^n\subset X$ such that for all scalars $(a_i)_{i=1}^n$, $$\max_{1\leqslant i\leqslant n}|a_i|\leqslant \|\sum_{i=1}^n a_ix_i\|\leqslant K \max_{1\leqslant i\leqslant n}|a_i|.$$  We let $I_\infty(X)=\sup_{K\geqslant 1}o(T_\infty(X,K))$.  We note that $I_\infty(X)=\infty$ if and only if $c_0$ embeds into $X$.  

We collect the following facts.  

\begin{proposition}\begin{enumerate}[(i)]\item For any Banach space $X$, $I_1(X)\geqslant I_\infty(X^*)$. \item For any Banach space $X$, $I_1(X)=1+\dim X$ if $X$ has finite dimension, $I_1(X)=\infty$ if $\ell_1$ embeds into $X$, and otherwise there exists an ordinal $\xi>0$ such that $I_1(X)=\omega^\xi$.  The analogous statement holds for the $I_\infty$ index.  \item For any limit ordinal $\xi$ and any Banach space $X$, $Sz(X), I_1(X), I_\infty(X)\neq \omega^{\omega^\xi}$.   \item Suppose $X$ has a $1$-unconditional basis.  Then $I_1(X)\geqslant Sz(X)$ and if $Sz(X)\geqslant \omega^\omega$, $Sz(X)=I_1(X)$.    \end{enumerate}

\label{proximity}
\end{proposition}

Items $(i)$, $(ii)$, and the portion of $(iii)$ dealing with $I_1$ and $I_\infty$  appeared in \cite{Causey1}.  The remaining items appeared in \cite{Causey}.   Items $(ii)$ and the portion of $(iii)$ dealing with $I_1$ and $I_\infty$ appeared for separable Banach spaces in \cite{JuddOdell}, while the remaining portion of $(iii)$ and item $(iv)$ appeared for separable Banach spaces in \cite{AJO}.

We remark here that for any sequence $(\xi_n)$ of ordinals, there exists a sequence of disjoint $B$-trees $(B_n)$ such that, with $T_n=B_n\cup \{\varnothing\}$,  $o(T_n)=\xi_n+1$ and such that $T_n$ is full.  Indeed, we may let $B_n$ denote those non-empty sequences $(n, \gamma_i, k_i)_{i=1}^p$ such that $\gamma_i\in [0, \xi_n)$, $\gamma_p<\ldots <\gamma_1$, and $k_i\in \nn$.  It is trivial to check by inducion that for any $\zeta$, $B_n^\zeta$ is the collection of all non-empty sequences $(n, \gamma_i, k_i)_{i=1}^p\in B_n$ such that $\gamma_p\geqslant \zeta$. This yields the correct order of $T_n$.   The inclusion of the $k_i$ terms ensures that $T_n$ is full.

We say that a subset $C$ of a tree $T$ is \emph{segment complete} if for every $s\preceq t$, $s,t\in C$, $\{u: s\preceq u\preceq t\}\subset C$.

Fix $\theta\in (0,1)$.  Fix an ordinal $\alpha$ and let $\xi=\omega^{\omega^\alpha}$.   Let $(B_n)$ be a sequence of disjoint $B$-trees such that with $T_n=B_n\cup \{\varnothing\}$, $o(T_n)=\xi^n+1$ and $T_n$ is full.  Let $B=\cup_{n=1}^\infty B_n$.    For each $n\in \nn$, $1\leqslant k\leqslant n$, and $\gamma\in [0, \xi^{n-k})$, let $$G_{n,k,\gamma}= B_n^{\xi^k \gamma}\setminus B_n^{\xi^k (\gamma+1)}.$$ Note that $G_{n, k, \gamma}$ is segment complete and that $(G_{n,k, \gamma})_{\gamma<\xi^{n-k}}$ partitions $B_n$.  Let $c_{00}(B)$ denote the span of the canonical Hamel basis $\{e_t: t\in B\}$ and let $e_t^*$ be the coordinate functional to $e_t$.   For each $n\in \nn$, $1\leqslant k\leqslant n$, and $\gamma\in [0, \xi^{n-k})$, let $$K_{n, k, \gamma}=\{0\}\cup \Bigl\{\theta^{k-1}\sum_{G_{n,k,\gamma}\ni s\preceq t} \ee_s e_s^*:\ee_s\in S_\mathbb{F}, t\in G_{n,k, \gamma}\Bigr\},$$ $$L_{n,k, \gamma}=|K_{n, k, \gamma}|,$$ $$K=\underset{n,k,\gamma}{\bigcup}K_{n, k, \gamma},\hspace{5mm} L=|K|.$$   Note that the basis $(e_t: t\in B)$ becomes a normalized, $1$-unconditional basis for the completion $\mathfrak{G}$ of $c_{00}(B)$ under the norm $\|x\|=\sup_{x^*\in K}|x^*(x)|$.    Note also that the norming set $K$ is contained in the unit ball $B_{\mathfrak{G}^*}$.   Furthermore, since for any $x\in \mathfrak{G}$, $$\max_{x^*\in B_{\mathfrak{G}^*}}\text{Re\ }x^*(x)=\|x\|=\max_{x^*\in K}\text{Re\ }x^*(x),$$ $B_{\mathfrak{G}^*}=\overline{\text{co}}^{w^*}(K)$ by Theorem \ref{hb}.

\begin{lemma}\begin{enumerate}[(i)]\item The collection $I=\{L_{n, k, \gamma}: n\in \nn, 1\leqslant k\leqslant n, \gamma<\xi^{n-k}\}$ satisfies the hypotheses of Claim \ref{trivial claim}.\item For any $\ee>0$, any $n\in \nn$, any $1\leqslant k\leqslant n$, and any $\gamma<\xi^{n-k}$, $Sz_\ee(L_{n, k, \gamma}) \leqslant \xi^k+1$.  \item Suppose $j\in \nn$ is such that  $\ee>2 \theta^j$. Then for any $k,n\in \nn$ such that  $j< k\leqslant n$ and for any ordinal $\gamma<\xi^{n-k}$, $Sz_\ee(L_{n, k, \gamma}) \leqslant \xi^j+1$. \item For any $n,j\in \nn$, $1\leqslant k\leqslant n$,  $\gamma<\xi^{n-k}$, and $\ee>2\theta^j$, $Sz_\ee(L_{n, k, \gamma})\leqslant \xi^j+1$.  \item If $0<\ee<\theta^{n-1}$, $Sz_\ee(L_{n, n, 0})=\xi^n+1$.

\end{enumerate}
\label{ebola}
\end{lemma}

\begin{proof} We recall that all trees are endowed with the coarse wedge topology.

$(i)$ Note that if $x^*\in L_{n,k, \gamma}$, $\supp(x^*)\subset G_{n, k, \gamma}$.  For any $s\in B$, there is a finite subset $F$ of $I$ such that if $x^*\in L\setminus \cup_{A\in F}A$, $x^*(e_s)=0$.  Indeed, if $s\in B_n$, then for each $1\leqslant k\leqslant n$, there exists a unique $\gamma_k$ such that $s\in G_{n,k,\gamma_k}$, and $F=\{L_{n,k,\gamma_k}:1\leqslant k\leqslant n\}$ has this property.  Therefore if $(x^*_\lambda)$ is any net in $L$ is such that for any $P\in I$, $(x^*_\lambda)$ is eventually in $\cup I\setminus (P\setminus\{0\})$, for any $s\in B$, $x^*_\lambda(e_s)\to 0$, whence $(x^*_\lambda)$ is $w^*$-null.  Next, suppose $(m,k,\gamma)\neq (n, l, \delta)$, $x^*\in L_{m,k, \gamma}\setminus\{0\}$, and $y^*\in L_{n,l, \delta}\setminus\{0\}$.  If $m\neq n$, $\supp(x^*)\cap \supp(y^*)\subset G_{m,k, \gamma}\cap G_{n, l, \delta}=\varnothing$ and $x^*\neq y^*$.  If $m=n$ and $k\neq l$, $$\{x^*(e_s):s\in B\}=\{0, \theta^{k-1}\}\neq \{0, \theta^{l-1}\}=\{y^*(e_s): s\in B\},$$ and $x^*\neq y^*$.  If $m=n$, $k=l$, and $\gamma\neq \delta$, then $\supp(x^*)\cap \supp(y^*)\subset G_{m,k, \gamma}\cap G_{m, k, \delta}=\varnothing$.  Thus $L_{m, k, \gamma}\cap L_{n, l, \delta}=\{0\}$.  It remains to show that $L_{n, k,\gamma}$ is $w^*$-compact.   Consider the following equivalence relation $\approx$ on $T_n^{\xi^ k\gamma}$.   We let $s\approx t$ if $s=t$ or if $s,t\in T_n^{\xi^k(\gamma+1)}$.   Let $q:T_n^{\xi^k \gamma}\to Q=T_n^{\xi^k \gamma}/\approx$ denote the quotient map sending a member of $T^{\xi^k \gamma}_n$ to its equivalence class.  Endow $Q$ with the quotient topology induced by this map, noting that $Q$ is compact.   Note that $T^{\xi^k \gamma}_n\setminus T^{\xi^k (\gamma+1)}_n=G_{n,k, \gamma}$, and the equivalence classes of $Q$ are $T_n^{\xi^k (\gamma+1)}$ and $\{t\}$, $t\in G_{n,k, \gamma}$.  We define the map $\varphi:Q\to L_{n,k, \gamma}$ by $\varphi(T_n^{\xi^k (\gamma+1)})=0$ and $\varphi(\{t\})=\theta^{k-1}\sum_{G_{n,k, \gamma}\ni s\preceq t} e_s^*$ for $t\in G_{n,k, \gamma}$.  It is straightforward to check that this is a bijection, and we need only to show that $\varphi$ is continuous to see that it is a homeomorphism.    Note that for any $\tau\in Q$, $\supp(\varphi(\tau))\subset G_{n,k, \gamma}$, thus we only need to show that for any net $(\tau_\lambda)\subset Q$ converging to $\tau\in Q$ and for any $s\in G_{n,k, \gamma}$, $\varphi(\tau_\lambda)(e_s)\to \varphi(\tau)(e_s)$. For every $\lambda$, fix $t_\lambda\in \tau_\lambda$ and fix $t\in \tau$.    Note that if $t\in G_{n,k, \gamma}$, then $t_\lambda\to t$ in the topology of $T^{\xi^k \gamma}_n$, and if $t\in T^{\xi^k \gamma}_n$, then for any $s\in G_{n, k, \gamma}$, eventually $s\not\preceq t_\lambda$.   In the case that $t\in G_{n,k, \gamma}$, for any $s\in G_{n,k, \gamma}$, $\varphi(\tau_\lambda)(e_s)\to \theta^{k-1}=\varphi(\tau)(e_s)$ if $s\preceq t$ and $\varphi(\tau_\lambda)(e_s)\to 0=\varphi(\tau)(e_s)$ if $s\not\preceq t$.  If $t\in T^{\xi^k (\gamma+1)}_n$, then $\varphi(\tau_\lambda)(e_s)\to 0=\varphi(\tau)(e_s)$ for every $s\in G_{n,k, \gamma}$.

$(ii)$ We claim that for any $n\in \nn$, $1\leqslant k\leqslant n$,  $\gamma<\xi^{n-k}$, $\ee>0$, and for any $\zeta\leqslant \xi^k$, $$s_\ee^\zeta(L_{n, k, \gamma})\subset \{0\}\cup \Bigl\{\theta^{k-1} \sum_{G_{n, k, \gamma}\ni s\preceq t} e_s^*: t\in G_{n, k, \gamma}\cap T_n^{\xi^k\gamma+\zeta}\Bigr\}.$$   Since $G_{n, k, \gamma}\cap T_n^{\xi^k (\gamma+1)}=\varnothing$, setting $\zeta=\xi^k$ yields that $s^{\xi^k}_\ee(L_{n, k, \gamma})\subset \{0\}$, and $Sz_\ee(L_{n, k, \gamma})\leqslant \xi^k+1$ as desired.   The base case and limit ordinal cases are clear.  Assume that $t_\lambda, t\in G_{n, k, \gamma}\cap T_n^{\xi^k\gamma+\zeta}$, $\theta^{k-1}\sum_{G_{n,k, \gamma}\ni s\preceq t_\lambda}e_s^*\neq \theta^{k-1}\sum_{G_{n, k, \gamma}\ni s\preceq t} e_s^*$ and $$\theta^{k-1}\sum_{G_{n,k}\ni s\preceq t_\lambda} e_s^*\underset{w^*}{\to}\theta^{k-1}\sum_{G_{n, k, \gamma}\ni s\preceq t} e_s^*.$$  It follows that eventually $t\prec t_\lambda$, whence $t\in G_{n, k, \gamma}\cap T_n^{\xi^k\gamma + \zeta+1}$, as desired.   

$(iii)$ Fix any $n,j\in \nn$, $j< k\leqslant n$, and suppose that $\ee>2\theta^j$.  We claim that for any $\zeta$, \begin{equation} s_\ee^\zeta(L_{n, k, \gamma})\subset \{0\}\cup \Bigl\{\theta^{k-1}\sum_{G_{n, k, \gamma}\ni s\preceq t} e_s^*: t\in G_{n, k, \gamma}\cap T_n^{\xi^k \gamma+\xi^{k-j}\zeta}\Bigr\}=:A_\zeta.\tag{1}\end{equation} Note that a representation of a functional as $\theta^{k-1}\sum_{G_{n,k, \gamma}\ni s\preceq t} e_s^*$ for $t\in G_{n, k, \gamma}\cap T_n^{\xi^k\gamma+\xi^{k-j}\zeta}$ is unique, and the set $G_{n, k, \gamma}\cap T_n^{\xi^k\gamma+\xi^{k-j}\zeta}$ is segment complete.    The base and limit ordinal cases are clear.  To obtain a contradiction, suppose $s_\ee^\zeta(L_{n, k, \gamma})\subset A_\zeta$ while $x^*\in s_\ee^{\zeta+1}(L_{n, k, \gamma})\setminus A_{\zeta+1}$.  By the inductive hypothesis, there exists $t\in G_{n, k, \gamma}\cap T_n^{\xi^k\gamma+\xi^{k-j}\zeta}$ such that $x^*=\theta^{k-1}\sum_{G_{n, k, \gamma}\ni s\preceq t}e_s^*$. Since $x^*\in s_\ee^{\zeta+1}(L_{n, k, \gamma})$, there exists a net $(x^*_\lambda)\subset s_\ee^\zeta(L_{n, k, \gamma})$ converging $w^*$ to $x^*$ and such that $\|x^*_\lambda-x^*\|>\ee/2>\theta^j$ for all $\lambda$.  By the inductive hypothesis, for each $\lambda$, there exists a unique $t_\lambda\in G_{n, k, \gamma}\cap T_n^{\xi^k\gamma + \xi^{k-j}\zeta}$ such that $x^*_\lambda=\theta^{k-1}\sum_{G_{n, k, \gamma}\ni s\preceq t_\lambda} e_s^*$.  It follows that $t\prec t_\lambda$ eventually, and we may assume $t\prec t_\lambda$ for all $\lambda$.   Since we have assumed $x^*\notin A_{\zeta+1}$, $t\notin T_n^{\xi^k\gamma+\xi^{k-j}(\zeta+1)}$.  Since $t\prec t_\lambda$, $t_\lambda\notin T_n^{\xi^k\gamma+\xi^{k-j}(\zeta+1)}$.  However, since $x^*, x^*_\lambda\in A_\zeta$, $t, t_\lambda\in T_n^{\omega^k \gamma+\xi^{k-j}\zeta}$, and therefore $t, t_\lambda\in T_n^{\xi^k\gamma+\xi^{k-j}\zeta}\setminus T_n^{\xi^k\gamma+\xi^{k-j}(\zeta+1)}$.  Since $$\xi^k\gamma + \xi^{k-j}\zeta= \xi^{k-j}\xi^j \gamma+\xi^{k-j}\zeta= \xi^{k-j}(\xi^j\gamma+\zeta)$$ and $$\xi^k\gamma+\xi^{k-j}(\zeta+1) = \xi^{k-j}(\xi^j\gamma+\zeta+1),$$ it follows that $$t, t_\lambda\in T_n^{\xi^k\gamma+\xi^{k-j}\zeta}\setminus T_n^{\xi^k\gamma+\xi^{k-j}(\zeta+1)} = T_n^{\xi^{k-j}(\xi^j\gamma+\zeta)}\setminus T_n^{\xi^{k-j}(\xi^j\gamma+\zeta+1)}= G_{n,\hspace{1mm} k-j, \hspace{1mm}\xi^j\gamma+\zeta}.$$  Fix any $\lambda$.   The set $\sigma=\{s: t\prec s\preceq t_\lambda\}$ is a subset of $\sigma_1:=\{s\in G_{n, k-j, \xi^j \gamma+\zeta}: s\preceq t_\lambda\}$, and $$x^*_\lambda-x^*= \theta^{k-1}\sum_{s\in \sigma}e_s^* = \theta^j\theta^{k-j-1}\sum_{s\in \sigma}e_s^*.$$  Since $\theta^{k-j-1}\sum_{s\in \sigma}e_s^*$ is a suppression projection of $\theta^{k-j-1}\sum_{s\in \sigma_1}e_s^*\in L_{n, k-j, \xi^j\gamma+\zeta}\subset B_{\mathfrak{G}^*}$, it follows that $$\theta^j<\|x^*_\lambda-x^*\| = \theta^j \|\theta^{k-j-1}\sum_{s\in \sigma} e_s^*\|\leqslant \theta^j \|\theta^{k-j-1}\sum_{s\in \sigma_1} e_s^*\|\leqslant\theta^j.$$   This contradiction yields successor case, and therefore completes the proof of $(1)$.  This claim easily yields $(iii)$.

$(iv)$ Fix any $n\in \nn$, $1\leqslant k\leqslant n$, and $\gamma<\xi^{n-k}$. Suppose that $\ee>2\theta^j$.  Then the result follows by $(ii)$ if $k\leqslant j$ and by $(iii)$ if $j<k$.   

$(v)$ Note that $$L_{n, n, 0}=\{0\}\cup \Bigl\{\theta^{n-1}\sum_{s\preceq t} e^*_s:  t\in B_n\Bigr\}.$$ Given $t\in T_n$, let $\varphi(t)=0$ if $t=\varnothing$ and $\varphi(t)=\theta^{n-1}\sum_{s\preceq t}e_s^*$ if $t\in B_n$.   As in $(i)$, this is a homeomorphism from $T_n$ to $L_{n,n,0}$ with its $w^*$-topology.     Since $T_n$ is full and has order $\xi^n+1$, and since $L_{n,n,0}$ is homeomorphic to $T_n$, the Cantor-Bendixson index of $L_{n,n,0}$ is $\xi^n+1$.  Thus for any $\ee>0$, $Sz_\ee(L_{n, n, 0})\leqslant \xi^n+1$, since $Sz_\ee(\cdot)$ can obviously not exceed the Cantor-Bendixson index of a set.  We claim that if $0<\ee<\theta^{n-1}$, $\varphi(T_n^\zeta)\subset s_\ee^\zeta(L_{n,n,0})$ for any $\zeta$.   Here, $T_n^\zeta$ can refer to either the $\zeta^{th}$ derived tree or the $\zeta^{th}$ Cantor-Bendixson derivative, since they coincide because $T_n$ is full.   The proof is by induction on $\zeta$, where the base and limit cases are trivial.  Suppose $t\in T_n^{\zeta+1}$.  Since $T_n$ is full, $t$ is the limit of some net $(t_\lambda)\subset T_n^\zeta$, where $t_\lambda^-=t$.  Then $\varphi(t_\lambda)\underset{w^*}{\to}\varphi(t)$, while $$\|\varphi(t_\lambda)-\varphi(t)\|= \theta^{n-1}\|e_{t_\lambda}^*\|=\theta^{n-1}>\ee.$$ Since $(\varphi(t_\lambda))\subset s_\ee^\zeta(L_{n,n,0})$, we deduce that $\varphi(t)\in s_\ee^{\zeta+1}(L_{n,n,0})$.   This gives the inclusion, and since $o(T_n)=\xi^n+1$, $Sz_\ee(L_{n,n,0})= \xi^n+1$.

\end{proof}

\begin{corollary} $Sz(\mathfrak{G})=I_1(\mathfrak{G})=I_\infty(\mathfrak{G}^*)=\omega^{\omega^{\alpha+1}}$.   

\label{ebola corollary}
\end{corollary}

\begin{proof} We first recall that $B_{\mathfrak{G}^*}=\overline{\text{co}}^{w^*}(K)$. Therefore by Theorem \ref{theorem}, we only need to compute $Sz(K)$ in order to compute $Sz(\mathfrak{G})$.

It follows from Lemma \ref{ebola}$(v)$ that $Sz(\mathfrak{G})\geqslant \sup_n \xi^n= (\omega^{\omega^\alpha})^\omega= \omega^{\omega^{\alpha+1}}$.  It follows from Lemma \ref{ebola}$(i)$ and $(iv)$ as well as Claim \ref{trivial claim} that $Sz(L)\leqslant \omega^{\omega^{\alpha+1}}$, since for any $\ee>0$, $\sup_{n,k, \gamma} Sz_\ee(L_{n, k, \gamma})\leqslant \xi^j+1<\omega^{\omega^{\alpha+1}}$, where $j\in \nn$ is such that $\ee>2\theta^j$.

It follows from Theorem \ref{theorem} and Lemma \ref{unconditionalization} that $$Sz(\mathfrak{G})=\Gamma(Sz(K))=\Gamma(Sz(L))=\omega^{\omega^{\alpha+1}}.$$ Since $\omega^{\alpha+1}\geqslant \omega$, it follows from Proposition \ref{proximity} that $$Sz(\mathfrak{G}) = I_1(\mathfrak{G})\geqslant I_\infty(\mathfrak{G}^*).$$   We need only show that $I_\infty(\mathfrak{G}^*)\geqslant \omega^{\omega^{\alpha+1}}$.   To that end, we claim that for any $n\in \nn$ and any $t\in G_{n,  n, 0}$,  $(e^*_{t|_i})_{i=1}^{|t|}\in T_\infty(\mathfrak{G}^*, \theta^{1-n})$, which follows from the fact that $\theta^{n-1}\sum_{i=1}^{|t|} e_{t|_i}^*\in L\subset B_{\mathfrak{G}^*}$ and the fact that $(e^*_{t|_i})_{i=1}^{|t|}$ is normalized and $1$-unconditional.  Since $o(G_{n,n,0})=\xi^n+1$, we deduce that $\sup_n o(T_\infty(\mathfrak{G}^*, \theta^{1-n}))\geqslant \sup_n \xi^n+1=\omega^{\omega^{\alpha+1}}$, whence the result follows.

\end{proof}

Fix  ordinals $\alpha, \beta$ such that $\beta$ has countable cofinality and $\alpha+\beta>\max\{\alpha, \beta\}$. Note that $\omega^\beta$ is also of countable cofinality.   Fix a sequence $0<\beta_1<\beta_2<\ldots$ such that $\beta_n\uparrow \omega^\beta$.  Let $(B_n)$ be a sequence of disjoint $B$-trees such that $o(B_n)=\omega^\alpha \beta_n$ and $T_n:=\{\varnothing\}\cup B_n$ is full.  Let $B=\cup_{n=1}^\infty B_n$.  Fix a sequence $(\theta_n)\subset (0,1)$ tending to $0$.  For each $n\in \nn$ and each $\gamma\in [0, \beta_n)$, let $G_{n, \gamma}=B_n^{\omega^\alpha \gamma}\setminus B_n^{\omega^\alpha(\gamma+1)}$. Note that $(G_{n, \gamma})_{\gamma<\beta_n}$ partitions $B_n$ into segment complete subsets.   Let $c_{00}(B)$ denote the span of the canonical Hamel basis $(e_t: t\in B)$ and let $e_t^*$ denote the coordinate functional to $e_t$.    For each $n\in \nn$ and $\gamma\in [0, \beta_n)$, let $$K_n=\{0\}\cup \Bigl\{\theta_n\sum_{\varnothing \prec s \preceq t} \ee_s e_s^*: \ee_s \in S_\mathbb{F}, t\in B_n\Bigr\}, \hspace{5mm} L_n=|K_n|,$$ $$K_{n, \gamma}=\{0\}\cup \Bigl\{\sum_{G_{n,\gamma}\ni s\preceq t} \ee_s e_s^*: \ee_s\in S_\mathbb{F}, t\in G_{n,\gamma}\Bigr\}, \hspace{5mm} L_{n, \gamma}=|K_{n, \gamma}|,$$ $$K=\Bigl(\underset{n=1}{\overset{\infty}{\bigcup}}K_n\Bigr)\cup \Bigl(\underset{n=1}{\overset{\infty}{\bigcup}}\underset{\gamma<\beta_n}{\bigcup} K_{n, \gamma}\Bigr), \hspace{5mm} L=|K|.$$    Let $\mathfrak{S}$ denote the completion of $c_{00}(B)$ under the norm $\|x\|=\sup_{x^*\in K}|x^*(x)|$.  Then $(e_t)_{t\in B}$ becomes a normalized, $1$-unconditional basis of $\mathfrak{S}$.    Furthermore, $K\subset B_{\mathfrak{S}^*}$ and $\overline{\text{co}}^{w^*}(K)=B_{\mathfrak{S}^*}$ by Theorem \ref{hb}.

\begin{lemma}\begin{enumerate}[(i)]\item The collection $I=\{L_n, L_{n,\gamma}: n\in \nn, \gamma<\beta_n\}$ satisfies the hypotheses of Claim \ref{trivial claim}. \item If $\ee>2\theta_n$, $Sz_\ee(L_n)< \omega^\beta$. \item If $0<\ee<\theta_n$, $Sz_\ee(L_n)=\omega^\alpha \beta_n+1$. \item For any $\ee>0$, any $n\in \nn$, and any $\gamma<\beta_n$, $Sz_\ee(L_{n, \gamma})\leqslant \omega^\alpha+1$.  \end{enumerate}

\end{lemma}

\begin{proof}$(i)$ Since any member of $L_n$ or $L_{n, \gamma}$ must have its support in $B_n$, which is disjoint from $B_m$ when $m\neq n$, $L_n\cap L_m=\{0\}$ when $n\neq m$.  Similarly, for any $\gamma_1$, $\gamma_2$, $L_m\cap L_{n, \gamma_2}=\{0\}$ and $L_{m, \gamma_1}\cap L_{n, \gamma_2}=\{0\}$ when $m\neq n$.   When $\gamma_1\neq \gamma_2<\beta_n$, $L_{n, \gamma_1}\cap L_{n, \gamma_2}=\{0\}$, since for any $x^*\in L_{n, \gamma_1}$, $\supp(x^*)\in G_{n, \gamma_1}$ which is disjoint from $G_{n, \gamma_2}$.   For any $s\in B$, there exists a unique pair $(n, \gamma)$ such that $s\in G_{n, \gamma}$.  Therefore if $x^*\in L$ is such that $x^*(e_s)\neq 0$, $x^*\in L_n\cup L_{n, \gamma}$.  Thus we deduce as in Lemma \ref{ebola} that if $(x^*_\lambda)\subset L$ is any net such that for any $P\in I$, $(x^*_\lambda)$ is eventually not in $P$, $x^*_\lambda\underset{w^*}{\to}0$.  Showing that $L_n$, $L_{n, \gamma}$ are $w^*$-compact again follow as in Lemma \ref{ebola}.

$(ii)$ Let $\varphi:T_n\to L_n$ denote the function $\varphi(\varnothing)=0$ and $\varphi(t)=\sum_{\varnothing\prec s\preceq t}e_s^*$.  Arguing as in $(i)$, we note that this is a homeomorphism.  We claim that if $\ee>2\theta_n$, then for any $0\leqslant \gamma<\beta_n$, $$s_\ee^\gamma(L_n) \subset \varphi(T_n^{\omega^\alpha \gamma}).$$   Since $\varnothing=T_n^{\omega^\alpha \beta_n+1}$, this will yield that $Sz_\ee(L_n)\leqslant \beta_n+1<\omega^\beta$.    We prove the claim by induction, with the base case and limit ordinal case being trivial.  Assume that $s_\ee^\gamma(L_n)\subset\varphi(T_n^{\omega^\alpha \gamma})$ and $x^*\in s_\ee^{\gamma+1}(L_n)$.   Then there exists a net $(x^*_\lambda)\subset s_\ee^\gamma(L_n)$ such that $x^*_\lambda\underset{w^*}{\to} x^*$ and $\|x^*_\lambda-x^*\|>\ee/2>\theta_n$ for all $\lambda$.    Let $t=\varphi^{-1}(x^*)\in T_n^{\omega^\alpha \gamma}$ and let $t_\lambda=\varphi^{-1}(x^*_\lambda)\in T_n^{\omega^\alpha \gamma}$.  If $t\in T_n^{\omega^\alpha(\gamma+1)}$, then $x^*\in \varphi(T_n^{\omega^\alpha (\gamma+1)})$ as desired, so assume $t\in T_n^{\omega^\alpha \gamma}\setminus T_n^{\omega^{\alpha}(\gamma+1)}=G_{n, \gamma}$.  We may assume $t\preceq t_\lambda$ for all $\lambda$, whence $t_\lambda\notin T_n^{\omega^\alpha(\gamma+1)}$, and it follows that $(t_\lambda)\subset G_{n,\gamma}$.  Fix $\lambda$.  From this it follows that $\sigma:=\{s: t\prec s\preceq t_\lambda\}$ is a subset of $\sigma_1:=\{s: G_{n,\gamma}\ni s\preceq t\}$ and $x^*_\lambda-x^*=\theta_n\sum_{s\in \sigma} e^*_s$.   Since $\theta_n\sum_{s\in \sigma_1} e_s^*\in \theta_n L$, and $f_\lambda-f$ is a suppression projection of $\theta_n \sum_{s\in \sigma_1} e_s^*\in L\subset B_{\mathfrak{S}^*}$, $$\|x^*_\lambda-x^*\| \leqslant \theta_n\|\sum_{s\in \sigma_1} e_s^*\|\leqslant \theta_n<\ee/2,$$ and this contradiction yields the result.

$(iii)$ This is the same as Lemma \ref{ebola}$(iv)$.

$(iv)$ We claim that for any $n\in \nn$ and any $\gamma<\beta_n$, $$s^\zeta_\ee(L_{n, \gamma})\subset \{0\}\cup \Bigl\{\theta_n \sum_{G_{n, \gamma}\ni s\preceq t} e_s^*: t\in G_{n, \gamma}\cap T_n^{\omega^\alpha \gamma + \zeta}\Bigr\}.$$  Note that since $G_{n, \gamma}\cap T_n^{\omega^\alpha (\gamma+1)}=\varnothing$, this claim will show that $s^{\omega^\alpha}_\ee(L_{n, \gamma})\subset \{0\}$, yielding $(iv)$.   The base and limit ordinal cases are trivial. Assume the $\zeta$ case holds and for some $(t_\lambda), t\in G_{n, k}\cap T^{\omega^\alpha \gamma+\zeta}$, $\theta_n\sum_{G_{n, \gamma}\ni s\preceq t} e_s^*\neq \theta_n \sum_{G_{n, \gamma}\ni s\preceq t_\lambda} e_s^*$ and $$\theta_n \sum_{G_{n, \gamma}\ni s\preceq t_\lambda} e_s^*\underset{w^*}{\to} \theta_n\sum_{G_{n, \gamma}\ni s\preceq t} e_s^*.$$  From this it follows that eventually $t\prec t_\lambda$, whence $t\in T_n^{\omega^\alpha \gamma+\zeta+1}$, and the result follows.

\end{proof}

\begin{corollary} $I_1(\mathfrak{S})\geqslant I_\infty(\mathfrak{S}^*) \geqslant Sz(\mathfrak{S})=\omega^{\alpha+\beta}$.  If $\alpha+\beta\geqslant \omega$, these inequalities are equalities.  

\end{corollary}

The proof is the same as the proof of Corollary \ref{ebola corollary}.   

\begin{corollary} Let $$\Gamma=\{\omega^\xi: \xi\in \textbf{\emph{Ord}}, \xi\text{\ has countable cofinality}\},$$ $$\Lambda=\{\omega^{\omega^\xi}: \xi\text{\ \emph{is a limit ordinal}}\}$$, $$\Omega=\Gamma\setminus \Lambda.$$   \begin{enumerate}[(i)]\item For any ordinal $\xi$, there exists a Banach space $X$ with $Sz(X)=\xi$ if and only if $\xi\in \Omega$.   \item For any ordinal $\xi$, there exists a Banach space $X$ with $I_1(X)=\xi$ if and only if $\xi\in \nn\cup \Omega$. \item For any ordinal $\xi$, there exists a Banach space $\xi$ with $I_\infty(X)=\xi$ if and only if $\xi\in \nn\cup \Omega$.   \end{enumerate}

\end{corollary}

\begin{proof} Let $\mathfrak{Sz}$ denote the class of ordinals $\xi$ such that there exists a Banach space $X$ with $Sz(X)=\xi$.  Let $\mathfrak{I}_1$ denote the class of ordinals $\xi$ such that there exists a Banach space $X$ with $I_1(X)=\xi$, and let $\mathfrak{I}_\infty$ be defined similarly.  Note that for any Banach space $X$, either $Sz(X)=1$, $Sz(X)=\infty$, or $Sz(X)=\sup_n Sz(B_{X^*}, 1/n)$ and $Sz(B_{X^*}, 1/n)<Sz(X)$.  This is because $Sz(B_{X^*}, 1/n)$ is a successor, while $Sz(X)$ is a limit.  From this we deduce that $Sz(X)$ must have countable cofinality if it is an ordinal.  Similarly, $I_1(X)$ is finite, $I_1(X)=\infty$, or $I_1(X)=\sup_n o(T_1(X, n))$, and $I_1(X)$ has countable cofinality.  The same is true of $I_\infty(X)$.  This fact combined with Proposition \ref{szlenk facts}$(v)$ yields that $\mathfrak{Sz}\subset \Gamma$, and combining this fact with Proposition \ref{proximity}, $\mathfrak{I}_1, \mathfrak{I}_\infty\subset \Gamma\cup \nn$.  Again using Proposition \ref{proximity},  $\mathfrak{Sz}, \mathfrak{I}_1, \mathfrak{I}_\infty\subset \Omega$. Thus we deduce one containment of each of $(i)$-$(iii)$.

For any natural number $n$, $I_1(\ell_2^{n-1})=I_\infty(\ell_2^{n-1})=n$, where $\ell_2^0=\{0\}$.   Thus $\nn\subset \mathfrak{I}_1, \mathfrak{I}_\infty$. Note that $Sz(\mathbb{F})=1$, so it remains only to show that  $\Omega\setminus\{1\}\subset \mathfrak{Sz}, \mathfrak{I}_1, \mathfrak{I}_\infty$.   

Fix $\xi\in \Omega\setminus\{1\}$. Then $\xi=\omega^\zeta$ for some $\zeta>0$, and since $\xi$ has countable cofinality, so does $\zeta$.   Write the Cantor normal form of $\zeta$ as $\zeta=\omega^{\gamma_1}n_1+\ldots + \omega^{\gamma_k} n_k$.  If $k=n_k=1$, $\xi=\omega^{\omega^\gamma}$ for some $\gamma$.  Note that $\gamma=0$ or $\gamma$ is a successor, otherwise $\xi\in \Lambda$. If $\gamma=0$, $\xi=\omega$, and $I_1(\ell_2)=I_\infty(\ell_2)=Sz(\ell_2)=\omega$.    Assume $\gamma$ is a successor, say $\gamma=\alpha+1$.  Then with $\mathfrak{G}$ defined as above, $$Sz(\mathfrak{G})=I_1(\mathfrak{G})=I_\infty(\mathfrak{G}^*)=\omega^{\omega^{\alpha+1}}=\xi.$$  Next, suppose that either $k>1$ or $n_k>1$.  Then let $\alpha=\omega^{\gamma_1}n_1+\ldots + \omega^{\gamma_k}(n_k-1)$ and $\beta= \omega^{\gamma_k}$. Note that $\alpha+\beta=\zeta$, and since $\zeta$ has countable cofinality, so does $\beta$.   By our computations above, $Sz(\mathfrak{S})=\omega^{\alpha+\beta}=\omega^\zeta$.  If $\zeta\geqslant \omega$,  $$I_1(\mathfrak{S})=I_\infty(\mathfrak{S}^*)=\omega^{\alpha+\beta}=\omega^\zeta=\xi.$$  Thus it remains to show the existence of  Banach spaces $X,Y$ with $I_1(X)=I_\infty(Y)=\omega^{n+1}$ for any $n\in \nn$.  Let $X_1=(\oplus_k\ell_1^k)_{\ell_2}$.  Assuming $X_n$ has been defined, let $X_{n+1}=(\oplus_k \ell_1^k(X_n))_{\ell_2}$.    It was shown in \cite{BCFW} that for each $n\in\nn$, $I_1(X_n)=\omega^{n+1}$ and $I_\infty(X^*_n)=\omega^{n+1}$.

\end{proof}

\section{Combinatorial necessities}

The proof of Theorem \ref{theorem} is a modification of an argument due to Schlumprecht \cite{Schlumprecht}.   Schlumprecht used the Schreier families and repeated averages hierarchy (first given in \cite{AMT}) to show that in order to check that all convex combinations of certain vectors are uniformly large, it is sufficient to check only over certain convex combinations the coefficients of which come from the repeated averages hierarchy.  Our proof differs from the proof given in \cite{Schlumprecht} in the following ways.  First, that proof used a characterization of the Szlenk index of a separable Banach space due to Alspach, Judd, and Odell \cite{AJO}.  Since we want to compute the Szlenk index of arbitrary $w^*$-compact sets, we use a different characterization of the Szlenk index from \cite{Causey} which was inspired by the result from \cite{AJO}.  Second, since we want to consider possibly uncountable ordinals and sets in the duals of non-separable Banach spaces, we must find suitable replacements for both the Schreier families and the repeated averages hierarchy, which we do in this section.  Last, we must prove that in order to show that all convex combinations of certain vectors are uniformly large, it is sufficient to check only over the convex combinations the coefficients of which come from our analogue of the repeated averages hierarchy, we must prove a combinatorial result, Theorem \ref{ultimatelemma}.

Given a sequence $t=(\xi_i)_{i=1}^n$ of ordinals and an ordinal $\xi$, we let $\xi+t=(\xi+\xi_i)_{i=1}^n$.  Given a set $G\subset [0, \zeta)^{<\nn}$, we let $\xi+G=\{\xi+t: t\in G\}\subset [\xi, \xi+\zeta)^{<\nn}$.    We next define the $B$-trees $\Gamma_\xi$, which will serve as the skeletons of our analogues of the Schreier families. For each $\xi\in \ord$, $\gr_\xi$ will be a $B$-tree on the set $[0, \omega^\xi)$ having order $\omega^\xi$.  We let $\gr_0=\{(0)\}$.    Next, assuming $\gr_\xi\subset [0, \omega^\xi)^{<\nn}$ has been defined, we define the $B$-trees $\gr_{\xi, n}$ for $n\in \nn$.   We let $\gr_{\xi, 1}=\gr_\xi$.  If $\gr_{\xi, n}$ has been defined, we let $$\gr_{\xi, n+1}=\bigl\{\omega^\xi n+s: s\in \gr_\xi\bigr\}\cup \bigl\{(\omega^\xi n+t)\cat u: s\in \gr_\xi, t\in MAX(\gr_\xi), u\in \gr_{\xi, n}\bigr\}$$  and $\gr_{\xi+1}=\cup_{n=1}^\infty \gr_{\xi, n}$.   Last, if $\xi$ is a limit ordinal and if $\gr_\zeta$ has been defined for each $\zeta<\xi$, we let $$\gr_\xi = \bigcup_{\zeta<\xi}( \omega^\zeta + \gr_{\zeta+1}).$$   We collect the following easy facts.  The proofs are either easy induction arguments or use standard facts about ordinals which can be found in \cite{Monk}.  

\begin{proposition}Let $\xi$ be an ordinal.  \begin{enumerate}[(i)]\item $\gr_\xi$ is a $B$-tree with order $\omega^\xi$.   \item $\omega^\xi+\gr_{\xi+1}\subset [\omega^\xi, \omega^{\xi+1})^{<\nn}$. \item If $\xi$ is a limit ordinal, the union $\gr_\xi=\cup_{\zeta<\xi}(\omega^\zeta + \gr_{\zeta+1})$ is a totally incomparable union.  \item If $(\xi_i)_{i=1}^p\in \gr_{\xi, n}$ for some $n\in\nn$, $\xi_1\in [\omega^\xi(n-1), \omega^\xi n)$. \item The union $\gr_{\xi+1}=\cup_{n=1}^\infty \gr_{\xi, n}$ is a totally incomparable union. \item For a given sequence $t$,  $t\in \gr_{\xi, n}$ for some $n\in \nn$ if and only if there exist $1\leqslant m\leqslant n$, $t_1, \ldots, t_m\in \gr_\xi$ so that $t_i$ is maximal in $\Gamma_\xi$ for each $1\leqslant i<m$ and so that $$t=(\omega^\xi(n-1)+ t_1)\cat(\omega^\xi(n-2)+t_2)\cat \ldots \cat(\omega^\xi(n-m)+ t_m),$$ and in this case $m$ and the sequences $t_i$ are unique.   Moreover, for this $t$, $t\in MAX(\gr_{\xi+1})$ if and only if $m=n$ and $t_n\in MAX(\gr_\xi)$.   \end{enumerate}

\label{longprop}
\end{proposition}

If $t\in \gr_{\xi, n}$, and if $t$ can be written as $(\omega^\xi(n-1)+t_1)\cat\ldots (\omega^\xi(n-m)+t_m)$ as in Proposition \ref{longprop}$(vi)$, we let $\iota_{\xi,n}(t)=t_m$,  $\lambda_{\xi,n}(t)=m$.  We say $t$ is on \emph{level} $m$.

Our next task is to define the analogue of the repeated averages hierarchy. Given a $B$-tree $T$, we let $\Pi T=\{(s,t)\in T\times MAX(T): s\preceq t\}$.    For each $\xi\in \ord$, we define $p_\xi:\Pi \Gamma_\xi\to [0,1]$ such that $\sum_{\varnothing \prec t'\preceq t}p_\xi(t',t)=1$ for each $t\in MAX(\gr_\xi)$.   We let $p_0((0), (0))=1$.  Next, assume $p_\xi$ has been defined.   For $t\in MAX(\gr_{\xi+1})$, $t\in \gr_{\xi,n}$ for some $n\in \nn$ and $t=(\omega^\xi(n-1)+t_1)\cat ( \omega^\xi(n-2)+t_2)\cat \ldots \cat t_n$ for some $t_1, \ldots, t_n\in MAX(\gr_\xi)$.    If $s\preceq t$, then for some $1 \leqslant m\leqslant  n$ and $s'\preceq t_m$, $$s=(\omega^\xi(n-1)+t_1)\cat \ldots \cat (\omega^\xi(n-m+1)+t_{m-1})\cat (\omega^\xi(n-m+1)+s'),$$ where the first $m-1$ sequences in the concatenation are understood to be omitted if $m=1$.   In this case, we define $$p_{\xi+1}(s,t)=\frac{1}{n}p_\xi(s', t_m).$$ With this definition, if $t=(\omega^\xi(n-1)+t_1)\cat \ldots \cat t_n\in MAX(\Gamma_{\xi, n})$ and if $s_i=(\omega^\xi(n-1)+t_1)\cat \ldots \cat (\omega^\xi (n-i)+t_i)$ and $s_0=\varnothing$, $$\sum_{\varnothing \prec s\preceq t} p_{\xi+1}(s, t)= \sum_{i=1}^n \sum_{s_{i-1} \prec s\preceq s_i} p_{\xi+1}(s, t) = \frac{1}{n}\sum_{i=1}^n \sum_{\varnothing \prec t'\preceq t_i} p_\xi(t', t_i)=1.$$    Last, assume $p_\zeta$ has been defined for  every $\zeta<\xi$, where $\xi$ is a limit ordinal.  If $(s,t)\in \Pi \Gamma_\xi$, then there exists a unique $\zeta<\xi$ and a unique $(s',t')\in \Pi\Gamma_{\zeta+1}$ such that $s=\omega^\zeta+s'$ and $t=\omega^\zeta+t'$.   Then we let $p_\xi(s,t)=p_{\zeta+1}(s',t')$. It follows by a trivial proof by induction that for any $\xi\in \ord$, if $s\in \gr_\xi$ and if $t, t'\in MAX(\gr_\xi)$ are any two maximal extensions of $s$, $p_\xi(s,t)=p_\xi(s, t')$.   We may therefore unambiguously define $\pr_\xi:\gr_\xi\to  [0,1]$ by $\pr_\xi(s)=p_\xi(s,t)$, where $t$ is a maximal extension of $s$ in $\Gamma_\xi$.  Note that since $\Gamma_\xi$ is well-founded, every $s\in \Gamma_\xi$ admits a maximal extension.  Note that for any $\xi\in \ord$ and any $n\in \nn$, if $t\in \gr_{\xi,n}$, $n \pr_{\xi+1}(t)=\pr_\xi(\iota_{\xi, n}(t))$.

In the remainder of this section,  $D$ will be a directed set. For any $B$-tree $T$ on a set $\Lambda$, we let $TD=\{(t, \sigma)\in (\Lambda\times D)^{<\nn}: t\in T\}$. That is, $TD=\{(t, (U_1, \ldots, U_n))\in (\Lambda\times D)^{<\nn}: U_i\in D, t\in T\}$.   Of course this is a $B$-tree, and one easily verifies by induction that for any ordinal $\xi$, $(TD)^\xi=T^\xi D$, so that $TD$ and $T$ have the same order. We let $\Omega_\xi=\Gamma_\xi D$ and $\Omega_{\xi, n}=\Gamma_{\xi, n}D$.   We define $\pr_\xi$ on $\Omega_\xi D$ by letting $\pr_\xi(t, \sigma)=\pr_\xi(t)$.  For $\xi\in \ord$ and $n\in \nn$, we define the function $\lambda_{\xi,n}:\Omega_{\xi, n}\to \{1, \ldots n\}$ by $\lambda_{\xi, n}(t, \sigma)=\lambda_{\xi, n}(t)$, where $\lambda_{\xi, n}$ defined on $\Gamma_{\xi, n}$ is as given following Proposition \ref{longprop}.  We let $\Lambda_{\xi,n,m}=\{\tau\in \gr_{\xi,n}: \lambda_{\xi, n}(\tau)=m\}$, $\Sigma_{\xi, n, m}=MAX(\Lambda_{\xi, n, m})$, and $\Sigma_{\xi, n}=\cup_{m=1}^n\Sigma_{\xi, n, m}$. Note that $\tau\in \Sigma_{\xi,n}$ if and only if $\iota_{\xi, n}(\tau)\in MAX(\Omega_\xi)$.    

If $(t, \sigma)\in \Omega_{\xi, n}$, and if $t=(\omega^\xi(n-1)+t_1)\cat\ldots \cat (\omega^\xi (n-m)+t_m)$ is the representation given in Proposition \ref{longprop}$(vi)$, we define $\iota_{\xi, n}(t, \sigma)=(t_m, \sigma')$, where $\sigma'$ is the tail of $\sigma$ which has the same length as $t_m$. We also define $\pi_{\xi, n}(t, \sigma)= ((\omega^\xi(n-1)+t_1)\cat\ldots \cat (\omega^\xi(m-2)+t_{m-1}), \sigma'')$, where $\sigma''$ is the initial segment of $\sigma$ having the same length as $t_1\cat\ldots \cat t_{m-1}$. We agree that in the case $m=1$, $\pi_{\xi,n}(t, \sigma)=\varnothing$.  The purpose of $\pi_{\xi,n}$ is to take $t$ on a certain level $m$ of $\Omega_{\xi,n}$ and return the maximal sequence on the $m-1^{st}$ level which $t$ extends (or if $t$ is on the first level, to return the empty sequence).  We let $U_\tau=\{\tau'\in \Omega_{\xi, n}: \pi_{\xi, n}(\tau)=\pi_{\xi, n}(\tau')\}$, and refer to this set as the \emph{unit containing} $\tau$. A set $U\subset \Omega_{\xi, n}$ will be called a \emph{unit} if $U=U_\tau$ for some $\tau$. It is clear that any unit $U$ of $\Omega_{\xi, n}$ is $U$ canonically identified with $\Omega_\xi$ via the map $\iota_{\xi, n}:U\to \Omega_\xi$.  This canonical identification is monotone, preserves immediate predecessors, and $\mathbb{P}_\xi(\iota_{\xi,n}(\tau))=n\mathbb{P}_{\xi+1}(\tau)$ for any $\tau\in U$.   We will implicitly make use of these facts throughout.  In particular, if $g:U\to \rr$ is a function and if $g':\Omega_\xi\to \rr$ is given by $g'(\iota_{\xi,n}(\tau))=g(\tau)$, $$\frac{1}{n}\sum_{\varnothing \prec \tau' \preceq \iota_{\xi,n}(\tau)} \mathbb{P}_\xi(\tau')g'(\tau')= \sum_{\pi_{\xi,n}(\tau)\prec \tau'\preceq \tau} \mathbb{P}_{\xi+1}(\tau')g(\tau')$$ for any $\tau\in U$.  This relationship between  averages is vital and implicitly used in the final section.

 If $T_1, T_2$ are any $B$-trees, and $S_1\subset T_1D$, $S_2\subset T_2D$, we say a function $\theta:S_1 \to S_2$ is a \emph{pruning} provided it is monotone and for every $\tau_1=(t_1, (U_1, \ldots, U_n))\in S_1$, $ U_n\leqslant_D V_m$, where  $\theta(\tau_1)=(t_1, (V_1, \ldots, V_m))$.     We say a pair $(\theta, e)$ is an \emph{extended pruning} if $\theta:S_1\to S_2$ is a pruning and if $e:MAX(S_1)\to MAX(S_2)$ is such that for any $\tau\in MAX(S_1)$, $\theta(\tau)\preceq e(\tau)$. We abuse notation and write $(\theta, e):S_1\to S_2$ rather than $(\theta, e):S_1\times MAX(S_1)\to S_2\times MAX(S_2)$.  For $\xi\in \ord$ and $m,n\in \nn$, we say a pair $(\theta, f)$ is a \emph{level pruning} provided that $\theta:\Omega_{\xi, m}\to \Omega_{\xi, n}$ is a pruning and $f:\Sigma_{\xi, m}\to \Sigma_{\xi, n}$ is such that \begin{enumerate}[(i)]\item for any $\tau\in \Omega_{\xi, m}$, $\theta(U_\tau)\subset U_{\theta(\tau)}$, \item  if $\tau'\in \Omega_{\xi, m}$ is such that $\tau\preceq \tau'\in MAX(U_\tau)$, $f(\tau')\in MAX(U_{\theta(\tau)})$ and $\theta(\tau)\preceq f(\tau')$, \item if $\tau'\in \Sigma_{\xi, m, i}$ for some $i<m$ and if $\tau$ is an extension of $\tau'$ with $\lambda_{\xi, m}(\tau)>i$, $f(\tau')\prec \theta(\tau)$. \end{enumerate}  This implies that for any unit $U$ of $\Omega_{\xi, m}$, $(\theta|_U, f|_{MAX(U)})$ is an extended pruning into a single unit of $\Omega_{\xi,n}$.  Note that if $m=n$ and if $(\theta, f):\Omega_{\xi, m}\to \Omega_{\xi, n}$ is any level pruning, $f$ maps $MAX(\Omega_{\xi, n})$ into itself.      Note that the composition of prunings is a pruning, and the analogous statement holds for  extended prunings and level prunings. 
 
 Note that if $T_1, T_2$ are $B$-trees and $\phi:T_1\to T_2$ is any monotone function such that $|\phi(t)|=|t|$ for all $t\in T_1$, then $\phi':T_1D\to T_2D$ given by $\phi'(t, \sigma)=(\phi(t), \sigma)$ defines a pruning.

 For any $\xi\in \ord$ and $m\leqslant n$, $m,n\in \nn$, $\Omega_{\xi, m}$ is naturally identifiable with several subsets of $\Omega_{\xi, n}$. We discuss these identifications here, since they will be used many times throughout.   The most obvious is the identification of $\Omega_{\xi, m}$ with the first $m$ levels of $\Omega_{\xi, n}$,  $\cup_{i=1}^m \Lambda_{\xi, n, i}$, given by $$\bigl((\omega^\xi(m-1)+t_1)\cat\ldots \cat (\omega^\xi(m-i)+t_i), \sigma\bigr)\mapsto \bigl((\omega^\xi(n-1)+t_1)\cat \ldots \cat (\omega^\xi (n-i)+t_i), \sigma\bigr).$$  Note that this function is a pruning, the pair consisting of this function and its restriction to $MAX(\Omega_{\xi, n})$ is an extended pruning, and the pair consisting of this function and its restriction to $\Sigma_{\xi, m}$ is a level pruning. In the sequel, we will refer to this identification as the natural identification of $\Omega_{\xi, m}$ with the first $m$ levels of $\Omega_{\xi, n}$.    Moreover, if $\tau\in \Sigma_{\xi, n, n- m}$, $\Omega_{\xi, m}$ is naturally identifiable with $A_\tau=\{\tau'\in \Omega_{\xi, n}: \tau\prec \tau'\}$ via the function $\Omega_{\xi, m}\to A_\tau$ given by $\tau'\mapsto \tau\cat \tau'$.    Note that this function is a pruning, the pair consisting of this function and its restriction to $MAX(\Omega_{\xi, m})$ is an extended pruning, and the pair consisting of this function and its restriction to $\Sigma_{\xi, m}$ is a level pruning.  We will refer to this as the natural identification of $\Omega_{\xi, m}$ with $A_\tau$.   Moreover, for any other $\tau_1\in \Sigma_{\xi, n,n-m}$, $A_\tau$ and $A_{\tau_1}$ are naturally identifiable, since each is naturally identifiable with $\Omega_{\xi, m}$.   We also note that $o(\Gamma_{\xi, m})=\omega^\xi m\leqslant o(\Gamma_{\zeta, n})$ if either $\xi<\zeta$ or $\xi=\zeta$ and $m\leqslant n$, and in this case there exists a monotone function $\phi:\Gamma_{\xi, m}\to \Gamma_{\zeta, n}$ such that $|\phi(t)|=|t|$ for all $t\in \Gamma_{\xi, m}$.  Then $\theta(t, \sigma)=(\phi(t), \sigma)$ is a pruning.  Therefore if $\zeta<\xi$ or if $\zeta=\xi$ and $m\leqslant n$, then there exists a pruning $\theta:\Omega_{\zeta, m}\to \Omega_{\xi, n}$.   Since $\Omega_{\xi, n}$ is well-founded, there exists $e:MAX(\Omega_{\zeta, m})\to MAX(\Omega_{\xi, n})$ such that $(\theta, e):\Omega_{\zeta, m}\to \Omega_{\xi, n}$ is an extended pruning. Indeed, for any $\tau\in MAX(\Omega_{\zeta, m})$, we may let $e(\tau)$ be any maximal extension of $\theta(\tau)$, at least one of which exists by well-foundedness.

For an ordinal $\xi$, $\ee\in \rr$, and a function $g:\Pi\Omega_\xi \to \rr$, we  say $g$ is \begin{enumerate}[(i)]\item $\ee$-\emph{small} if there exists an extended pruning $(\theta, e):\Omega_\xi\to \Omega_\xi$ such that for every $\tau\in MAX(\Omega_\xi)$, $\sum_{\varnothing\prec \tau'\preceq e(\tau)} \mathbb{P}_\xi(\tau')g(\tau', e(\tau))<\ee$, \item $\ee$-\emph{large} if for every $\delta>0$, there exists an extended pruning $(\theta, e):\Omega_\xi \to \Omega_\xi $ so that for every $(\tau', \tau)\in \Pi \Omega_\xi$, $g(\theta(\tau'), e(\tau))\geqslant \ee-\delta$.  \end{enumerate}

\begin{remark} Note that these alternatives are not exclusive.  Indeed, for any $\xi>0$, there exists a function $g:\Pi\Omega_\xi\to [-1,1]$ which is $1$-large and $\ee$-small for any $\ee>-1$.  One may define $g:\Pi\Omega_{\xi+1}\to [-1,1]$ by letting $f|_{\Pi \Omega_{\xi, n}}=(-1)^n$, which is both $1$-large and $\ee$-small for every $\ee>-1$.    If $\xi$ is a limit ordinal and there exists a function $f_\zeta:\Pi\Omega_{\zeta+1}\to [-1,1]$ which is $1$-large and $\ee$-small for every $\ee>-1$, we may define such a function $f:\Pi\Omega_\xi\to [-1,1]$ by identifying $(\omega^\zeta+\Gamma_{\zeta+1})D$ with $\Omega_{\zeta+1}$ and using this identification and $f_\zeta$ to define $f$ on $\Pi (\omega^\zeta+\Gamma_{\zeta+1})D$.

\end{remark}

\begin{theorem} For any ordinal $\xi$, any function $g:\Pi \Omega_\xi \to [-1,1]$,  and any $\ee\in \rr$, $g$ is either $\ee$-large or $\ee$-small.

\label{ultimatelemma}

\end{theorem}

The assumption that $g$ maps into $[-1,1]$ in Theorem \ref{ultimatelemma} is purely a matter of convenience.  It is clear how to deduce from Theorem \ref{ultimatelemma} the following generalization.    

\begin{corollary} For any ordinal $\xi$, if $g:\Pi\Omega_\xi\to \rr$ is any bounded function and if $\ee\in \rr$, then $g$ is either $\ee$-large or $\ee$-small.   
\label{conv}
\end{corollary}

\begin{corollary} For any ordinal $\xi$, if $g:\Pi\Omega_\xi \to \rr$ is a bounded function and $\ee\in \rr$ is such that for every $\tau\in MAX(\Omega_\xi)$, $\sum_{\varnothing\prec \tau'\preceq \tau} \mathbb{P}_\xi(\tau')g(\tau', \tau)\geqslant \ee$, then $g$ is $\ee$-large.  

\label{maincorollary}

\end{corollary}

Corollary \ref{maincorollary} follows from the fact that no such function $g$ could be $\ee$-small.

The proof of Theorem \ref{ultimatelemma}, of course, will be by induction.  The only case which will require some work is the successor case.  For that, we need the following technical piece.  

\begin{lemma} Suppose that $\xi\in \emph{\ord}$ is an ordinal such that for every $g:\Pi \Omega_\xi \to [-1,1]$ and any $\ee\in \rr$, $g$ is either $\ee$-large or $\ee$-small.  For any $n\in \nn$, $\ee\in \rr$, $\delta>0$, and $g:\Pi \Omega_{\xi, n}\to [-1,1]$,  there exist $p,q\in \nn_0$ such that $p+q=n$,  an extended pruning $(\theta, e):\Omega_{\xi, p}\to \Omega_{\xi, n}$, a level pruning $(\phi, f):\Omega_{\xi, q}\to \Omega_{\xi, n}$ and a function $e':MAX(\Omega_{\xi, q})\to MAX(\Omega_{\xi, n})$ such that \begin{enumerate}[(i)] \item for every $(\tau', \tau)\in \Pi \Omega_{\xi, p}$, $g(\theta(\tau'), e(\tau))\geqslant \ee-\delta$, \item for every $\tau\in MAX(\Omega_{\xi, q})$, $f(\tau)\preceq e'(\tau)$,  \item for every $\tau\in \Sigma_{\xi, q}$, $$\sum_{\pi_{\xi, n}(f(\tau))\prec \tau'\preceq  f(\tau)} \mathbb{P}_\xi(\iota_{\xi, n}(\tau'))g(\tau', e'(\tau''))< \ee$$ whenever $\tau''\in MAX(\Omega_{\xi, q})$ is a  maximal extension of $\tau$. \end{enumerate} Here, if $p=0$ (resp.  $q=0$), we take $\Omega_{\xi, 0}=\varnothing$ and agree that the conclusion on $(\theta, e)$ (resp. $(\phi, f)$) is vacuously satisfied.

\label{jobdone}
\end{lemma}

We will relegate the somewhat technical proof of Lemma \ref{jobdone} to the final section.

\begin{proof}[Proof of Theorem \ref{ultimatelemma}]The proof is by induction.   The base case follows from the fact that the function $g$ can be identified with the net $(g((0), (U)))_{U\in D}$.  Then there exists a subnet of this net which is either always in $[\ee, \infty)$ or always in $(-\infty, \ee)$.  This means that for every $U\in D$, we may select $V_U\in D$ with $U\leqslant_D V_U$ such that the pair $(\theta, e):\Omega_0\to \Omega_0$ given by $\theta((0), (U))=e((0), (U))=((0), (V_U))$ is an extended pruning which witnesses that $g$ is $\ee$-large provided that the subnet is always in $[\ee, \infty)$ and which witnesses that $g$ is $\ee$-small if the subnet is always in $(-\infty, \ee)$.   

Next, assume $\xi$ is a limit ordinal and the result holds for every $\zeta<\xi$. Fix a function $g:\Pi\Omega_\xi\to [-1,1]$.  Let $\Theta_{\zeta+1}=(\omega^\zeta+\Gamma_{\zeta+1})D$ for each $\zeta<\xi$.  Recall that $\Gamma_\xi=\cup_{\zeta<\xi}(\omega^\zeta + \Gamma_{\zeta+1})$, and therefore $\Omega_\xi=\cup_{\zeta<\xi} \Theta_{\zeta+1}$.  Recall that $\Theta_{\zeta+1}$ can be naturally identified with $\Omega_{\zeta+1}$ by $j_\zeta:\Omega_{\zeta+1}\to \Theta_{\zeta+1}$ given by $j_\zeta(t, \sigma)=(\omega^\zeta+t, \sigma)$. Moreover, with this identification, $\mathbb{P}_{\zeta+1}(\tau)=\mathbb{P}_\xi(j_\zeta(\tau))$ for every $\tau\in \Omega_{\zeta+1}$.   Using this identification, we may treat $g|_{\Pi\Theta_{\zeta+1}}$ as a function on $\Pi\Omega_{\zeta+1}$. Using the inductive hypothesis, this function is either $\ee$-large or $\ee$-small.  Let $\Sigma$ denote the set of those $\zeta<\xi$ such that the function $g|_{\Pi\Theta_{\zeta+1}}$ is $\ee$-large.  Then either $\sup \Sigma=\xi$ or $\sup [0, \xi)\setminus \Sigma=\xi$.  We will show that in the first case, $g$ is $\ee$-large, and otherwise $g$ is $\ee$-small.  Suppose $\sup \Sigma=\xi$ and fix $\delta>0$.   For each $\zeta\in \Sigma$, we find an extended pruning $(\theta_\zeta, e_\zeta):\Theta_{\zeta+1}\to \Theta_{\zeta+1}$ such that for every $(\tau', \tau)\in \Pi\Theta_{\zeta+1}$, $g(\theta_\zeta(\tau'), e_\zeta(\tau))\geqslant \ee-\delta$.   For every $\zeta<\xi$, fix $\eta_\zeta\in \Sigma$ such that $\zeta<\eta_\zeta$ and fix an extended pruning $(\theta'_\zeta, e'_\zeta):\Theta_{\zeta+1}\to \Theta_{\eta_\zeta+1}$.    Let $\theta|_{\Theta_{\zeta+1}}=\theta_{\eta_\zeta}\circ \theta'_\zeta$ and $e|_{MAX(\Theta_{\zeta+1})}= e_{\eta_\zeta}\circ e'_\zeta$.    One easily checks that $g(\theta(\tau'), e(\tau))\geqslant \ee-\delta$ for every $(\tau', \tau)\in \Pi\Omega_\xi$, and $g$ is $\ee$-large.  Next, suppose $\sup \Sigma<\xi$.   For every $\zeta\in [0,\xi)\setminus \Sigma$, since $f|_{\Theta_{\zeta+1}}$ is $\ee$-small by the inductive hypothesis, we may find an extended pruning $(\theta_\zeta, e_\zeta):\Theta_{\zeta+1}\to \Theta_{\zeta+1}$ such that for every $\tau\in MAX(\Theta_{\zeta+1})$, $\sum_{\varnothing \prec \tau'\preceq \tau} \mathbb{P}_\xi(\tau')g(\tau', e(\tau))<\ee$.   For every $\zeta<\xi$, fix $\eta_\zeta\in [0,\xi)\setminus \Sigma$ such that $\zeta<\eta_\zeta$ and fix an extended pruning $(\theta'_\zeta, e'_\zeta):\Theta_{\zeta+1}\to \Theta_{\eta_\zeta+1}$.  Let $\theta|_{\Theta_{\zeta+1}}=\theta_{\eta_\zeta}\circ \theta'_\zeta$ and $e|_{MAX(\Theta_{\zeta+1})}=e_{\eta_\zeta}\circ e_\zeta'$.  One easily checks that $(\theta, e):\Omega_\xi\to \Omega_\xi$ witnesses that $g$ is $\ee$-small.

Last, assume the result holds for some $\xi$.  Fix $\ee\in\rr$ and $g:\Pi\Omega_{\xi+1}\to [-1,1]$. Assume that $g$ is not $\ee$-small.    Of course, if $\ee>1$, $g$ is trivially $\ee$-small, so it must be that $\ee\leqslant 1$.      Fix $\delta>0$ and choose $\rho\in (0,1)$ such that $1-\rho + \rho(\ee-\delta/2)<\ee$.   For each $n\in \nn$, apply Lemma \ref{jobdone} to $g|_{\Pi\Omega_{\xi, n}}$ to obtain $p_n, q_n$ with $p_n+q_n=n$, extended prunings $(\theta_n, e_n):\Omega_{\xi, p_n}\to \Omega_{\xi, n}$, level prunings $(\phi_n, f_n):\Omega_{\xi, q_n}\to \Omega_{\xi, n}$ and functions $e'_n:MAX(\Omega_{\xi, q_n})\to MAX(\Omega_{\xi, n})$  satisfying the conclusions of Lemma \ref{jobdone} with $\ee$ replaced by $\ee-\delta/2$ and $\delta$ replaced by $\delta/2$.   We claim that if $q_n> \rho n$ for infinitely many $n\in \nn$, then $g$ is $\ee$-small.  Suppose that $n_1<n_2<\ldots$, $n_i\in \nn$,  are such that for all $i\in \nn$, $i\leqslant \rho n_i < q_{n_i}$.   For each $i\in \nn$, fix some extended pruning $(\phi''_i, e_i''):\Omega_{\xi, i}\to \Omega_{\xi, q_{n_i}}$ and define $\phi:\Omega_{\xi+1}\to \Omega_{\xi+1}$ by letting $\phi|_{\Omega_{\xi, i}}=\phi_{n_i}\circ\phi''_i$. Define $e:MAX(\Omega_{\xi+1})\to MAX(\Omega_{\xi+1})$ by letting $e=e'_{n_i}\circ e''_i$. Fix any $\tau'\in MAX(\Omega_{\xi, i})$ and let $\tau=e(\tau')$.   Let $\tau_1, \ldots, \tau_{n_i}$ be the initial segments of $\tau_{n_i}$ such that $\tau_j\in \Sigma_{\xi, n_i, j}$ for $1\leqslant j\leqslant n_i$. Let $\tau_0=\varnothing$.  Note that for each $1\leqslant j\leqslant n_i$, $\pi_{\xi, n_i}(\tau_j)=\tau_{j-1}$.   By condition $(iii)$ of Lemma \ref{jobdone}, $$\sum_{\tau_{j-1}\prec \tau''\preceq \tau_j} \mathbb{P}_\xi(\iota_{\xi, n_i}(\tau''))g(\tau'', \tau)=\sum_{\tau_{j-1}\prec \tau''\preceq \tau_j} \mathbb{P}_\xi(\iota_{\xi, n_i}(\tau''))g(\tau'', e'_{n_i}(e''_i(\tau')))<\ee-\delta/2$$ for all $j$ such that $\tau_j\in f_{n_i}(\Sigma_{\xi, q_{n_i}})$.  Of course, there are $q_{n_i}$ such values of $j$.  Let $A\subset\{1, \ldots, n_i\}$ denote the set of $j$ such that $\tau_j\in f_{n_i}(\Sigma_{\xi, q_{n_i}})$.     Then \begin{align*} \sum_{\tau''\preceq e(\tau')} \mathbb{P}_{\xi+1}(\tau'')g(\tau'', e(\tau')) &  = \frac{1}{n_i}\sum_{j=1}^{n_i}\sum_{\tau_{j-1}\prec \tau''\preceq \tau_j} \mathbb{P}_\xi(\iota_{\xi, n_i}(\tau''))g(\tau'', e(\tau')) \\ & \leqslant \frac{|A|}{n_i}(\ee-\delta/2) + \frac{n_i-|A|}{n_i} < \rho(\ee-\delta/2) + 1-\rho< \ee.  \end{align*} Since $\tau'\in MAX(\Omega_{\xi+1})$ was arbitrary, the extended pruning $(\phi, e)$ implies that $g$ is $\ee$-small.  This contradiction shows that $q_n\leqslant  n \rho$ for all but finitely many, and therefore all sufficiently large, $n\in \nn$.  This means that $p_n\geqslant (1-\rho) n$ for all sufficiently large $n\in \nn$, and we may fix $n_1<n_2<\ldots$ such that $i< (1-\rho) n_i\leqslant p_{n_i}$ for all $i\in \nn$.  We then fix any extended prunings $(\theta''_i, e''_i):\Omega_{\xi, i}\to \Omega_{\xi,p_{n_i}}$ and define $(\theta, e):\Omega_{\xi+1}\to \Omega_{\xi+1}$ by letting $\theta|_{\Omega_{\xi, i}}= \theta_{n_i}\circ\theta''_i$ and by letting $e|_{MAX(\Omega_{\xi, i})}= e_{n_i}\circ e''_i$. It follows from the construction that for all $(\tau', \tau)\in \Pi\Omega_{\xi+1}$, $g(\theta(\tau'), e(\tau))\geqslant \ee-\delta/2-\delta/2=\ee-\delta$.  Since $\delta>0$ was arbitrary, $g$ is $\ee$-large.

\end{proof}

\section{Proof of Theorem \ref{theorem}}

Let $X$ be a Banach space.  In this section, $D$ will be a weak neighborhood basis at $0$ in $X$ directed by reverse inclusion. Let $\Omega_\xi=\Gamma_\xi D$ as defined in the previous section. Given $\hhh\subset B_X^{<\nn}$, let $(\hhh)'_w$ denote those $t\in \hhh$ such that for any $U\in D$, there exists $x\in U$ such that $t\cat (x)\in \hhh$.   We define $(\hhh)_w^0=\hhh$, $(\hhh)_w^{\xi+1}=((\hhh)_w^\xi)_w'$, and if $\xi$ is a limit ordinal, let $(\hhh)_w^\xi=\cap_{\zeta<\xi} (\hhh)_w^\zeta$. Let $o_w(\hhh)$ denote the minimum ordinal $\xi$ such that $(\hhh)_w^\xi=\varnothing$ if such an ordinal exists, and otherwise write $o_w(\hhh)=\infty$.     For $\ee>0$, let $$\hhh^K_\ee=\{\varnothing\}\cup \bigl\{(x_i)_{i=1}^n\in B_X^{<\nn}: (\exists x^*\in K)(\forall 1\leqslant i\leqslant n)(\text{Re\ }x^*(x_i)\geqslant \ee)\bigr\}.$$   We will use the following characterization of Szlenk index.  

\begin{theorem}\cite{Causey} Let $K\subset X^*$ be $w^*$-compact and non-empty.  Let $\xi$ be an ordinal.  \begin{enumerate}\item $Sz(K)>\xi$ if and only if there exists $\ee>0$ such that $o_w(\hhh^K_\ee)>\xi$.  \item For any $\ee>0$, $o_w(\hhh^K_\ee)>\xi$ if and only if for any $B$-tree $T$ with $o(T)=\xi$, there exists a collection $(x_t)_{t\in TD}\subset B_X$ such that \begin{enumerate}[(i)]\item for any $\tau=\tau_1\cat (t,U)\in TD$, $x_\tau\in U$, \item for every $\tau\in TD$, $(x_{\tau|_i})_{i=1}^{|\tau|}\in \hhh_\ee^K$. \end{enumerate} \end{enumerate}

\label{characterization theorem}

\end{theorem}

Item $(2)$ of Theorem \ref{characterization theorem} was not stated in this way in \cite{Causey}.  It was proved for a particular choice of $B$-tree $T$ with $o(T)=\xi$.  However, it is easy to see that if there exists one $B$-tree $S$ with $o(S)=\xi$ and a collection $(y_\tau)_{\tau\in SD}$ satisfying items $(i)$ and $(ii)$ in $(2)$, then for any $B$-tree $T$ with $o(T)=\xi$, there exists such a collection $(x_\tau)_{\tau\in TD}$ satisfying items $(i)$ and $(ii)$.   Indeed, if we fix any monotone $\theta:T\to S$ which preserves lengths, recalling that at least one such function $\theta$ exists, then we may define $\phi:TD\to SD$ by $\phi((t,\sigma))=(\theta(t), \sigma)$.   Then we may define $(x_\tau)_{\tau\in TD}$ by $x_\tau=y_{\phi(\tau)}$, and observe that $(x_\tau)_{\tau\in TD}$ also satisfies items $(i)$ and $(ii)$ of $(2)$ from Theorem \ref{characterization theorem}.    

Since $\Gamma_\xi$ is a $B$-tree with $o(\Gamma_\xi)=\omega^\xi$, we obtain the following particular case of Theorem \ref{characterization theorem}.  We isolate it, since it is the particular form we will use.  We recall that $\Omega_\xi=\Gamma_\xi D$.

\begin{corollary} Let $K\subset X^*$ be a $w^*$-compact, non-empty set.  Then for any $\ee>0$, $o_w(\hhh^K_\ee)>\omega^\xi$ if and only if there exists $(x_\tau)_{\tau\in \Omega_\xi}\subset B_X$ such that \begin{enumerate}[(i)]\item for each $\tau=\tau_1\cat (t, U)\in \Omega_\xi$, $x_\tau\in U$, \item for each $\tau\in \Omega_\xi$, $(x_{\tau|_i})_{i=1}^{|\tau|}\in \hhh^K_\ee$.    \end{enumerate} In particular, $Sz(K)>\omega^\xi$ if and only if for some $\ee>0$, there exists a collection $(x_\tau)_{\tau\in \Omega_\xi}\subset B_X$ satisfying $(i)$ and $(ii)$.   

\label{my corollary}
\end{corollary}

The next corollary indicates that if $K$ is balanced, in order to verify the properties above, up to a small change in $\ee$, we only need to check how well $K$ norms certain convex combinations of the branches of a tree as in Corollary \ref{my corollary}. 

\begin{corollary} Suppose $K\subset X^*$ is $w^*$-compact, balanced, non-empty.  Fix $0<\ee_0<\ee$. Let $(x_\tau)_{\tau\in \Omega_\xi}\subset B_X$ be such that for every $\tau=\tau_1\cat (t,U)\in \Omega_\xi$, $x_\tau\in U$.  For each $\tau\in MAX(\Omega_\xi)$, let $$y_\tau=\sum_{\varnothing\prec \tau'\preceq \tau} \mathbb{P}_\xi(\tau')x_{\tau'}\in \text{\emph{co}}(x_{\tau'}: \varnothing\prec \tau'\preceq \tau).$$ If for each $\tau\in MAX(\Omega_\xi)$, there exists $y^*_\tau\in K$ such that $\text{\emph{Re\ }}y^*_\tau(y_\tau)\geqslant \ee$, then $o_w(\hhh^K_{\ee_0})>\omega^\xi$.   
\label{our corollary}
\end{corollary}

\begin{proof} Since $K$ is bounded, the function $g:\Pi\Omega_\xi\to \rr$ given by $g(\tau', \tau)=\text{Re\ }y^*_\tau(x_{\tau'})$ is bounded.  Since for any $\tau\in MAX(\Omega_\xi)$, $\sum_{\varnothing\prec \tau'\preceq \tau}\mathbb{P}_\xi(\tau')g(\tau', \tau) = \text{Re\ }y^*_\tau(y_\tau)\geqslant \ee$, Corollary \ref{maincorollary} implies that $g$ must be $\ee$-large.   Let $\delta=\ee-\ee_0$ and fix an extended pruning $(\theta, e):\Omega_\xi\to \Omega_\xi$ such that for every $(\tau', \tau)\in \Pi\Omega_\xi$, $g(\theta(\tau'), e(\tau))\geqslant \ee-\delta=\ee_0$.   Then by Corollary \ref{my corollary}, $(z_{\tau'})_{\tau'\in \Omega_\xi}=(x_{\theta(\tau')})_{\tau'\in \Omega_\xi}\subset B_X$ witnesses that $o_w(\hhh^K_{\ee_0})>\omega^\xi$.   Indeed, for any $\tau'=\tau_1\cat (t, U)\in \Omega_\xi$, since $\theta$ is a pruning, $\theta(\tau')=\tau_2\cat(t_2, V)$ for some $V\in D$ such that $V\subset U$.  Then $$z_{\tau'}=x_{\theta(\tau')}\in V\subset U,$$ showing that $(z_{\tau'})_{\tau'\in \Omega_\xi}$ satisfies $(i)$ of Corollary \ref{my corollary}.   For any $\tau\in MAX(\Omega_\xi)$, $y^*_{e(\tau)}$ is such that for each $1\leqslant i\leqslant |\tau|$, $\text{Re\ }y^*_{e(\tau)}(z_{\tau|_i})= y^*_{e(\tau)}(x_{\theta(\tau|_i)})\geqslant \ee_0$, since $(\tau|_i, \tau)\in \Pi\Omega_\xi$.  This shows that $(z_{\tau'})_{\tau'\in\Omega_\xi}$ satisfies $(ii)$ of Corollary \ref{my corollary}.

\end{proof}

\begin{proof}[Proof of Theorem \ref{theorem}] If $K$ is empty, there is nothing to show.  Assume $K$ is non-empty and let $L=\overline{\text{co}}^{w^*}(K)$.   By Proposition \ref{szlenk facts}$(v)$, we know $Sz(K)\leqslant Sz(L)$ and, since $Sz(L)$ must be $\infty$ or a gamma number, $\Gamma(Sz(K))\leqslant Sz(L)$.    We must show the reverse inequality.  Note that if we can show that $Sz(\overline{\text{co}}^{w^*}(S_\mathbb{F}K))\leqslant \Gamma(Sz(S_\mathbb{F}K))$, then since $L\subset \overline{\text{co}}^{w^*}(S_\mathbb{F}K)$, Lemma \ref{lemma1} will yield that $$Sz(L)\leqslant Sz(\overline{\text{co}}^{w^*}(S_\mathbb{F} K))\leqslant \Gamma(Sz(S_\mathbb{F} K))=\Gamma(Sz(K)),$$ which is the desired inequality.  Therefore we may replace $K$ with $S_\mathbb{F} K$ and assume that $K$ is balanced.

We will show that if $o_w(\hhh^L_\ee)>\omega^\xi$ for some $\ee>0$, then $o_w(\hhh^K_{\ee_0})>\omega^\xi$ for any $0<\ee_0<\ee$.   By the last statement of Corollary \ref{my corollary}, this will imply that if $Sz(L)>\omega^\xi$, $Sz(K)>\omega^\xi$, which implies that $Sz(L)\leqslant \Gamma(Sz(K))$.   Suppose $o_w(\hhh^L_\ee)>\omega^\xi$.  Fix $(x_\tau)_{\tau\in \Omega_\xi}\subset B_X$ as in Corollary \ref{my corollary}.  Since $L$ is $w^*$-compact, convex, and balanced, a finite sequence $(x_i)_{i=1}^n\subset B_X$ lies in $\hhh^L_\ee$ if and only if for every convex combination $x$ of $(x_i)_{i=1}^n$, $|x|_L=\max_{x^*\in L}\text{Re\ }x^*(x)\geqslant \ee$.   This follows from using the Hahn-Banach theorem as explained in \cite{Causey}.    For each $\tau\in MAX(\Omega_\xi)$, set $y_\tau=\sum_{\varnothing\prec \tau'\preceq \tau}\mathbb{P}_\xi(\tau')x_{\tau'}$ as in Corollary \ref{our corollary}.  Then our previous remarks guarantee that  $|y_\tau|_L\geqslant \ee$ for each $\tau\in MAX(\Omega_\xi)$.  But since $L$ is the $w^*$-closed, convex hull of $K$, and since $K$ is balanced, this means there exists $y^*_\tau\in K$ such that $\text{Re\ }y^*_\tau(y_\tau)\geqslant \ee$.   Applying Corollary \ref{our corollary} yields that $o_w(\hhh^K_{\ee_0})>\omega^\xi$ for any $0<\ee_0<\ee$.

\end{proof}

\section{Technical lemmata and the proof of Lemma \ref{jobdone}}

\begin{lemma} For any $\xi\in \emph{\ord}$, any $n\in \nn$,  any finite set $A$, and any $h:MAX(\Omega_{\xi, n})\to A$, there exist $a\in A$ and a level pruning $(\theta, f):\Omega_{\xi, n}\to \Omega_{\xi, n}$ such that for any $\tau\in MAX(\Omega_{\xi, n})$, $h(f(\tau))=a$.  
\label{easylemma}
\end{lemma}

\begin{proof} Recall that if $n=1$, a level pruning $(\theta, f):\Omega_{\xi,n}\to \Omega_{\xi,n}$ is simply an extended pruning. We recall also that any level pruning from $\Omega_{\xi,n}$ into itself must map $MAX(\Omega_{\xi,n})$ into itself.   We prove the result by induction on $\ord\times \nn$ ordered lexicographically.  The result for $\xi=0$ and $n=1$ follows from the fact that the net $(h((0), (U)))_{U\in D}$ in $A$ must have a constant subnet.  Assume $\xi$ is a limit ordinal and the result holds for every pair $(\zeta, 1)$ with $\zeta<\xi$.  Fix $h:MAX(\Omega_{\xi, 1})\to A$.  Since $\Omega_{\xi,1}=\cup_{\zeta<\xi} \Theta_{\zeta+1}$, where $\Theta_{\zeta+1}$ is as in the proof of Theorem \ref{jobdone} and can be naturally identified with $\Omega_{\zeta+1}=\Omega_{\zeta+1,1}$, we may apply the inductive hypothesis to obtain for each $\zeta<\xi$ some extended pruning $(\theta_\zeta, e_\zeta):\Theta_{\zeta+1}\to \Theta_{\zeta+1}$ and $a_\zeta\in A$ such that $h\circ e_\zeta(\tau)=a_\zeta$ for all $\tau\in MAX(\Theta_{\zeta+1})$.  Fix some $a\in A$ such that $\sup \{\zeta<\xi: a_\zeta=a\}=\xi$, noting that such an $a\in A$ must exist.  For each $\zeta<\xi$, we fix $\eta_\zeta>\zeta$ such that $a=a_{\eta_\zeta}$ and an extended pruning $(\theta'_\zeta, e'_\zeta):\Theta_{\zeta+1}\to \Theta_{\eta_\zeta+1}$. Let $(\theta,f):\Omega_\xi\to \Omega_\xi$ be defined by $\theta|_{\Theta_{\zeta+1}}= \theta_{\eta_\zeta}\circ \theta'_\zeta$ and $f|_{MAX(\Theta_{\zeta+1})}=e_{\eta_\zeta}\circ e'_\zeta$.    Then $a=h\circ f(\tau)$ for all $\tau\in MAX(\Omega_\xi)$.   

Assume that for some $\xi\in \ord$ and some $n\in \nn$,  the result holds for every pair $(\xi, k)$ with $1\leqslant k\leqslant n$.  Fix $h:MAX(\Omega_{\xi, n+1})
\to A$.   For each $\tau\in \Sigma_{\xi, n+1, 1}$, let $B_\tau=\{\tau'\in \Omega_{\xi, n+1}: \tau\prec \tau'\}$.  Then $B_\tau$ can be naturally identified with $\Omega_{\xi, n}$ in a way which also identifies $h|_{MAX(B_\tau)}$ with an $A$-valued function on $MAX(\Omega_{\xi, n})$.  Then there exists $a_\tau\in A$ and a level pruning $(\theta_\tau, f_\tau):B_\tau\to B_\tau$ so that for all $\tau'\in MAX(B_\tau)$, $h\circ f_\tau(\tau')=a_\tau$.  Of course, level prunings were defined as being between $\Omega_{\xi, p}$ and $\Omega_{\xi, q}$, but we agree to call $(\theta_\tau, f_\tau)$ a level pruning since it is identified with one via the identification of $B_\tau$ with $\Omega_{\xi, n}$.   We define $H:\Sigma_{\xi, n+1, 1}\to A$ by $H(\tau)=a_\tau$.   Note that since $\Lambda_{\xi, n+1, 1}$ is naturally identified with $\Omega_\xi$ and $\Sigma_{\xi, n+1, 1}=MAX(\Lambda_{\xi, n+1, 1})$  is naturally identified with $MAX(\Omega_\xi)$, we apply the inductive hypothesis again to obtain an extended pruning $(\theta', e'):\Lambda_{\xi, n+1, 1}\to \Lambda_{\xi, n+1, 1}$ and $a\in A$ so that $H\circ e'(\tau)=a$ for all $\tau\in \Sigma_{\xi, n+1, 1}$.   We then define $\theta$ on $\Lambda_{\xi, n+1, 1}$ by setting it equal to $\theta'$.  We define the level pruning $f$ on $\Sigma_{\xi, n+1, 1}$ by setting it equal to $e'$.   Given $\tau\in \Sigma_{\xi, n+1, 1}$, we first let $j_\tau:B_\tau\to B_{e'(\tau)}$ denote the natural identification of these sets and then define $\theta$ on $B_\tau$ by setting it equal to $\theta_{e'(\tau)}\circ j_\tau$.  In order to define $f$ on $B_\tau\cap \Sigma_{\xi, n+1}$, we note that $j_\tau$ identifies $B_\tau\cap \Sigma_{\xi, n+1}$ with $B_{e'(\tau)}\cap \Sigma_{\xi, n+1}$ and let $f|_{B_\tau}= f_{e'(\tau)}\circ j_\tau$.   Then $(\theta, f)$ is a level pruning and for any $\tau\in MAX(\Omega_{\xi, n+1})$, $a=h\circ f(\tau)$.

Last, assume that for some $\xi\in \ord$ and all $n\in \nn$, the result holds for every pair $(\xi, n)$.  Fix a function $h:MAX(\Omega_{\xi+1})\to A$.   For each $n\in\nn$, apply the inductive hypothesis to obtain a level pruning $(\theta_n, f_n):\Omega_{\xi, n}\to \Omega_{\xi, n}$ and $a_n\in A$ such that $a_n=h\circ f_n(\tau)$ for all $\tau\in MAX(\Omega_{\xi, n})$.   Fix $a\in A$ and natural numbers $n_1<n_2<\ldots$ such that for all $i\in \nn$, $a=a_{n_i}$.  Let $j_i:\Omega_{\xi, i}\to \Omega_{\xi, n_i}$ be the natural identification of $\Omega_{\xi, i}$ with the first $i$ levels of $\Omega_{\xi, n_i}$ and let $e'_i:MAX(\Omega_{\xi, i})\to MAX(\Omega_{\xi, n_i})$ be any function such that $(j_i, e'_i): \Omega_{\xi, i}\to \Omega_{\xi, n_i}$ is an extended pruning.    Let $\theta|_{\Omega_{\xi, i}}=\theta_{n_i}\circ j_i$ and $f|_{MAX(\Omega_{\xi, i})}= f_{n_i}\circ e'_i$.  

\end{proof}

\begin{lemma} Suppose that $\xi\in\emph{\ord}$ is such that for any $\ee\in \rr$ and any $g:\Pi\Omega_\xi \to [-1,1]$, $g$ is either $\ee$-large or $\ee$-small.  Then for any $\delta>0$, any $\ee\in \rr$, any $n\in \nn$, and any $g:\Pi \Omega_{\xi, n}\to [-1,1]$, there exists a level pruning $(\theta, f):\Omega_{\xi, n}\to \Omega_{\xi, n}$ such that for every unit $U$ of $\Omega_{\xi, n}$, either \begin{enumerate}[(i)]\item for every $\tau\in MAX(U)$ and every $\tau''\in MAX(\Omega_{\xi,n})$ such that $\tau\preceq \tau''$, $$\sum_{\pi_{\xi, n}(f(\tau)) \prec \tau'\preceq f(\tau)} \mathbb{P}_\xi(\iota_{\xi, n}(\tau'))g(\tau', f(\tau''))<\ee,$$ or \item  for every $\tau'\in U$ and every $\tau\in MAX(\Omega_{\xi, n})$ such that $\tau'\preceq \tau$, $$g(\theta(\tau'), f(\tau))\geqslant \ee-\delta.$$  \end{enumerate}   

\label{easy2}

\end{lemma}

\begin{proof} We induct on $n\in \nn$.  The $n=1$ case is simply  the hypothesis that for every $\ee\in \rr$, every function $g:\Pi \Omega_\xi\to [-1,1]$ is either $\ee$-large or $\ee$-small.   Assume the result holds for every natural number $k\leqslant n$.  Fix $\ee\in \rr$, $\delta>0$, and $g:\Pi \Omega_{\xi, n+1}\to [-1,1]$.    For each $\tau\in \Sigma_{\xi, n+1, 1}$, let $B_\tau=\{\tau'\in \Omega_{\xi, n+1}: \tau\prec\tau'\}$.   Identifying $B_\tau$ with $\Omega_{\xi, n}$, we apply the inductive hypothesis to obtain $(\theta_\tau, f_\tau):B_\tau\to B_\tau$ such that for every unit $U$ of $B_\tau$, either $(i)$ or $(ii)$ of the lemma is satisfied.   Let $(\tau_i)_{i=1}^r$ be a list of the initial segments of $\tau$ and let $A$ be a partition of $[-1,1]^r$ into sets of diameter not exceeding $\delta/2$, where $[-1,1]^r$ is endowed with the $\ell_\infty^r$ metric.   Define $h_\tau:MAX(B_\tau)\to A$ by letting $h_\tau(\tau')$ denote the member $S$ of $A$ such that $(g(\tau_i,f_\tau(\tau')))_{i=1}^r \in S$.   Still identifying $B_\tau$ with $\Omega_{\xi, n}$, Lemma \ref{easylemma} implies the existence of a level pruning $(\theta'_\tau, f'_\tau):B_\tau\to B_\tau$ and $S_\tau\in A$ such that for all $\tau'\in MAX(B_\tau)$, $h_\tau(f'_\tau(\tau'))=S_\tau$.  This means that for any $\tau_0\in \Lambda_{\xi, n+1, 1}$, any extension $\tau\in \Sigma_{\xi, n+1, 1}$ of $\tau_0$, and any two maximal extensions $\tau', \tau''\in MAX(\Omega_{\xi, n+1})$ of $\tau$, $|g(\tau_0, f_\tau\circ f'_\tau(\tau'))- g(\tau_0, f_\tau\circ f'_\tau(\tau'))|\leqslant \delta/2$.    

Let $G:\Pi \Lambda_{\xi, n+1, 1}\to [-1,1]$ be given by $$G(\tau_0, \tau)=\inf\{g(\tau_0, f_\tau\circ f'_\tau(\tau')): \tau\preceq \tau'\in MAX(\Omega_{\xi, n+1})\}.$$ By the last sentence of the previous paragraph, for any $(\tau_0, \tau)\in \Pi \Lambda_{\xi, n+1, 1}$ and any maximal extension $\tau'\in MAX(\Omega_{\xi, n+1})$ of $\tau$, $g(\tau_0, f_\tau\circ f'_\tau(\tau'))\leqslant \delta/2+G(\tau_0, \tau')$.    Since $\Lambda_{\xi, n+1, 1}$ is identifiable with $\Omega_\xi$, we may deduce that $G$ is either $(\ee-\delta/2)$-small or $(\ee-\delta/2)$-large.   This means that we may find an extended pruning $(\theta', e'):\Lambda_{\xi, n+1, 1}\to \Lambda_{\xi, n+1, 1}$ so that either $\sum_{\tau'\preceq e'(\tau)} \mathbb{P}_\xi(\iota_{\xi, n+1}(\tau'))G(\tau', e'(\tau)) < \ee-\delta/2$ for every $\tau\in \Sigma_{\xi, n+1, 1}$ or so that $G(\theta'(\tau'), e'(\tau))\geqslant \ee-\delta/2-\delta/2=\ee-\delta$ for every $(\tau', \tau)\in \Pi \Lambda_{\xi, n+1, 1}$.   

We define $\theta$ on $\Lambda_{\xi, n+1, 1}$ by setting it equal to $\theta'$ and $f$ on $\Sigma_{\xi, n+1, 1}$ by setting it equal to $e'$.  For a fixed $\tau\in \Sigma_{\xi, n+1, 1}$, let $j_\tau:B_\tau\to B_{e'(\tau)}$ be the natural identification. Define $\theta$ on $B_\tau$ by $\theta=\theta_{e'(\tau)}\circ \theta'_{e'(\tau)}\circ j_\tau$ and define $f$ on $B_\tau \cap \Sigma_{\xi, n+1}$ by $f=f_{e'(\tau)}\circ f'_{e'(\tau)}\circ j_\tau$.    We show that $(\theta, f)$ satisfies the conclusion.  We first show that one of the two alternatives is satisfied for the unit $U=\Lambda_{\xi, n+1, 1}$.   Suppose that we are in the case that for every $\tau\in \Sigma_{\xi, n+1, 1}=MAX(U)$, $$\sum_{\tau'\preceq e'(\tau)} \mathbb{P}_\xi(\iota_{\xi, n+1}(\tau'))G(\tau', e'(\tau))<\ee-\delta/2.$$ Then for any maximal extension $\tau''\in MAX(\Omega_{\xi, n+1})$ of $\tau$, $j_\tau(\tau'')$ is a maximal extension of $e'(\tau)$.  As mentioned in the previous paragraph, for every $\tau'\preceq e'(\tau)$, $$g(\tau', f(\tau''))= g(\tau', f_{e'(\tau)}\circ f'_{e'(\tau)}(j_\tau(\tau''))) \leqslant \delta/2 + G(\tau', e'(\tau)).$$   Since $\sum_{\tau'\preceq e'(\tau)} \mathbb{P}_\xi(\iota_{\xi, n+1}(\tau'))=1$, \begin{align*} \sum_{\pi_{\xi, n+1}(f(\tau))\prec \tau'\preceq f(\tau)} \mathbb{P}_\xi(\iota_{\xi, n+1}(\tau'))g(\tau', f(\tau'')) & = \sum_{\tau'\preceq e'(\tau)} \mathbb{P}_\xi(\iota_{\xi, n+1}(\tau'))(\delta/2+ G(\tau', e'(\tau))) \\ & = \delta/2+\sum_{\tau'\preceq e'(\tau)}\mathbb{P}_\xi(\iota_{\xi, n+1}(\tau'))G(\tau', e'(\tau)) \\ & <\delta/2+\ee-\delta/2=\ee.\end{align*} This shows that in this case, the unit $\Lambda_{\xi, n+1, 1}$ satisfies $(ii)$ in the statement of the lemma.  Next, suppose we are in the case that for every $(\tau', \tau)\in \Pi \Lambda_{\xi, n+1, 1}$, $G(\theta'(\tau), e'(\tau))\geqslant \ee-\delta$.   Then for any $\tau'\in \Lambda_{\xi, n+1, 1}$ and any $\tau''\in MAX(\Omega_{\xi, n+1})$, if $\tau\in \Sigma_{\xi, n+1,1}$ is the extension of $\tau'$ which is an initial segment of $\tau''$, then since $j_\tau(\tau'')$ is a maximal extension of $e'(\tau)$, \begin{align*} g(\theta(\tau'), f(\tau'')) & = g\bigl(\theta'(\tau), f_{e'(\tau)}\circ f'_{e'(\tau)}(j_\tau(\tau''))\bigr) \\ & \geqslant \inf \bigl\{g\bigl(\theta'(\tau'), f_{e'(\tau)}\circ f'_{e'(\tau)}(\hat{\tau})\bigr): e'(\tau)\preceq \hat{\tau}\in MAX(\Omega_{\xi, n+1})\bigr\} \\ & =G(\theta'(\tau'), e'(\tau)) \geqslant \ee-\delta. \end{align*} In this case, item $(i)$ in the statement of the lemma is satisfied by the unit $\Lambda_{\xi, n+1, 1}$.  Therefore in either case, one of the two items $(i)$, $(ii)$ is satisfied by this unit.

We finally show that one of the two items $(i)$, $(ii)$ is satisfied by every unit $U$ of $\Omega_{\xi, n+1}$ other than $\Lambda_{\xi, n+1, 1}$.   To that end, fix such a unit and fix $\tau\in \Sigma_{\xi, n+1, 1}$ such that $U\subset B_\tau$.   Since $(\theta_{e'(\tau)}'\circ j_\tau,f_{e'(\tau)}'\circ j_\tau):B_\tau\to B_{e'(\tau)}$ is (identified with) a level pruning, there exists a unit $V\subset B_{e'(\tau)}$ such that $$\theta_{e'(\tau)}'\circ j_\tau(U)\subset V \text{\ \ and\ \ }f_{e'(\tau)}'\circ j_\tau(MAX(U))\subset MAX(V).$$  By our choice of $(\theta_{e'(\tau)}, f_{e'(\tau)})$, either for every $\tau_1\in MAX(V)$ and every maximal extension $\tau_2$ of $\tau_1$, $$\sum_{\pi_{\xi, n+1}(f_{e'(\tau)}(\tau_1))\prec \tau'\preceq f_{e'(\tau)}(\tau_1)}\mathbb{P}_\xi(\iota_{\xi, n+1}(\tau'))g(\tau', f_{e'(\tau)}(\tau_2))<\ee,$$ or for every $(\tau_1, \tau_2)\in \Pi B_{e'(\tau)}$, $$g(\theta_{e'(\tau)}(\tau_1), f'_{e'(\tau)}(\tau_2))\geqslant \ee-\delta.$$  In the first case, for any $\tau_0\in MAX(U)$ and every maximal extension $\tau''$ of $\tau_0$, $\tau_1:=f'_{e'(\tau)}\circ j_\tau(\tau_0)\in MAX(V)$ and $\tau_2:=f'_{e'(\tau)}\circ j_\tau(\tau'')$ is a maximal extension of $\tau_1$, whence \begin{align*} \sum_{\pi_{\xi, n+1}(f(\tau_0))\prec \tau'\preceq f(\tau_0)} \mathbb{P}_\xi(\iota_{\xi, n+1}(\tau')) g(\tau', f(\tau'')) & = \sum_{\pi_{\xi, n+1}(f_{e'(\tau)}(\tau_1))\prec \tau'\preceq f_{e'(\tau)}(\tau_1)} \mathbb{P}_\xi(\tau')g(\tau', f_{e'(\tau)}(\tau_2)) \\ &  <\ee,\end{align*} and $(i)$ is satisfied for the unit $U$.  

In the second case, for any $\tau'\in U$ and any maximal extension $\tau''$ of $\tau'$, $\tau_1:=f'_{e'(\tau)}\circ j_\tau(\tau')\in V$ and $\tau_2:=f'_{e'(\tau)}\circ j_\tau(\tau'')$ is a maximal extension of $\tau_1$, whence $$g(\theta(\tau'), f(\tau'')) \geqslant g(\theta'(\tau_1), f_{e'(\tau)}(\tau_2))\geqslant \ee-\delta,$$ and $(ii)$ is satisfied for the unit $U$.    

\end{proof}

For $\xi\in \ord$ and $n\in \nn$, let us say that $P\subset \Sigma_{\xi, n}$ is \emph{unital} provided that for any unit $U$ in $\Omega_{\xi, n}$, either $MAX(U)\subset P$ or $P\cap MAX(U)=\varnothing$.  Note that the complement of a unital set is also unital.

\begin{lemma} For any $\xi\in \emph{\ord}$ and $n\in \nn$, if $P,Q\subset \Sigma_{\xi, n}$ is a partition of $\Sigma_{\xi, n}$ into unital subsets, then there exist $p,q\in \nn_0$ with $p+q=n$ and level prunings $(\theta, f):\Omega_{\xi, p}\to \Omega_{\xi, n}$,  $(\phi, f'):\Omega_{\xi, q}\to \Omega_{\xi, n}$ such that $f(\Sigma_{\xi, p})\subset P$ and $f'(\Sigma_{\xi, q})\subset Q$. As in Lemma \ref{jobdone}, we agree that $\varnothing = \Omega_{\xi, 0}$. 

\label{easy3}

\end{lemma}

\begin{proof} We induct on $n\in\nn$. Since $\Omega_{\xi,1}$ has only one unit, for $n=1$, either $\Sigma_{\xi, 1}=P$ or $\Sigma_{\xi, 1}=Q$.   In the first case, we take $p=1$ and let $(\theta, f)$ be the identity.  Otherwise we let $q=1$ and let $(\phi, f')$ be the identity.  

Assume the result holds for a given $n$ and let $P, Q$ be a partition of $\Sigma_{\xi, n+1}$ into unital subsets.  For each $\tau\in \Sigma_{\xi, n+1, 1}$, identify $B_\tau=\{\tau'\in \Omega_{\xi, n+1}: \tau\prec \tau'\}$ with $\Omega_{\xi, n}$ and note that this identification also identifies $P\cap B_\tau$, $Q\cap B_\tau$ with a partition of $\Sigma_{\xi, n}$ into unital sets.  Using this identification and the inductive hypothesis, we find $p_\tau, q_\tau\in \nn_0$ with $p_\tau+q_\tau=n$ and level prunings $(\theta_\tau, f_\tau):\Omega_{\xi, p_\tau}\to B_\tau$, $(\phi_\tau, f'_\tau):\Omega_{\xi, q_\tau}\to B_\tau$ such that $f_\tau(\Sigma_{\xi, p_\tau})\subset B_\tau \cap P$ and $f'_\tau(\Sigma_{\xi, q_\tau})\subset B_\tau \cap Q$.   Note that either $\Sigma_{\xi, n+1, 1}\subset P$ or $\Sigma_{\xi, n+1, 1}\subset Q$, since $\Lambda_{\xi, n+1, 1}$ is a single unit.  We assume that $\Sigma_{\xi, n+1, 1}\subset P$, with the other case being identical.  Let $p=1+\min_{\tau\in \Sigma_{\xi, n+1,1}} p_\tau$ and $q=\max_{\tau\in \Sigma_{\xi, n+1, 1}} q_\tau=n+1-p$.   We define $(\phi, f'):\Omega_{\xi, q}\to \Omega_{\xi, n+1}$ by setting it equal to $(\phi_\tau, f'_\tau)$ for some $\tau\in \Sigma_{\xi, n+1, 1}$ such that $q_\tau=q$. Of course, by this construction, $f_\tau(\Sigma_{\xi, q})\subset Q$.  Next, we define $(\theta, f):\Omega_{\xi, p}\to \Omega_{\xi, n+1}$.   Let $\theta|_{\Lambda_{\xi, p, 1}}$ be the natural  identification of  $\Lambda_{\xi, p, 1}$ with $\Lambda_{\xi, n+1, 1}$ and let $f|_{\Sigma_{\xi, p, 1}}$ be the natural identification of $\Sigma_{\xi, p, 1}$ with $\Sigma_{\xi, n+1,1}$.  If $p=1$, this completes the definition of $(\theta, f)$, and $f(\Sigma_{\xi, p})=f(\Sigma_{\xi, 1,1})\subset P$.  Otherwise fix $\tau\in \Sigma_{\xi, p, 1}$, let $C_\tau=\{\tau'\in \Omega_{\xi, p}: \tau\prec \tau'\}$, and let $j_\tau:C_\tau\to \Omega_{\xi, p_{f(\tau)}}$ be the natural identification of $C_\tau$ (which can be naturally identified with $\Omega_{\xi, p-1}$) with the first $p-1$ levels of $\Omega_{\xi, p_{f(\tau)}}$.  This may be done, since $p-1\leqslant p_{f(\tau)}$.   Then let $\theta|_{C_\tau}= \theta_{f(\tau)}\circ j_\tau$ and $f|_{C_\tau\cap \Sigma_{\xi,p}}= f_{f(\tau)}\circ j_\tau$.   Suppose that $\tau'\in \Sigma_{\xi, p}$.  If $\tau'\in \Sigma_{\xi, p,1}$, $f(\tau')\in P$ since $f(\Sigma_{\xi, p,1})=\Sigma_{\xi, n+1,1}\subset P$.  Otherwise there exists $\tau\in \Sigma_{\xi, p, 1}$ such that $\tau\prec \tau'$, and $\tau'\in C_\tau$.  Then $j_\tau(\tau')\in C_{f(\tau)}\cap \Sigma_{\xi, n+1}$, and our choice of $f_{f(\tau)}$ gives that $f(\tau')\in f_{f(\tau)}(C_{f(\tau)}\cap \Sigma_{\xi, n+1})\subset P$.   

\end{proof}

\begin{proof}[Proof of Lemma \ref{jobdone}] Suppose $\xi\in \ord$ is such that for every $g:\Omega_\xi\to [-1,1]$ and $\ee\in \rr$, $g$ is either $\ee$-large or $\ee$-small.    Fix $\ee\in \rr$, $\delta>0$, $n\in \nn$, and $g:\Pi \Omega_{\xi, n}\to [-1,1]$.  Fix a level pruning $(\theta'', f'')$ as in Lemma \ref{easy2}.   We define $P\subset \Sigma_{\xi, n}$ as follows: For a unit $U$ in $\Omega_{\xi, n}$,  $MAX(U)\subset P$ if for every $\tau'\in U$ and every extension $\tau''\in MAX(\Omega_{\xi,n})$ of $\tau'$, $g(\theta''(\tau'), f''(\tau''))\geqslant \ee-\delta$, and otherwise $MAX(U)\cap P=\varnothing$.    Let $Q=\Sigma_{\xi, n}\setminus P$. Then $P,Q$ is a partition of $\Sigma_{\xi, n}$ into unital subsets.  Note that the conclusion of Lemma \ref{easy2} implies that for any unit $U$ such that $MAX(U)\cap P=\varnothing$, for any $\tau\in MAX(U)$ and any maximal extension $\tau''\in MAX(\Omega_{\xi, n})$ of $\tau$, $\sum_{\pi_{\xi, n}(f''(\tau))\prec \tau'\preceq f''(\tau)}\mathbb{P}_\xi(\iota_{\xi, n}(\tau'))g(\tau', f''(\tau''))<\ee$.   Let $p,q$, $(\phi_p, f_p)$, and $(\phi_q, f_q)$ be as in the conclusion of Lemma \ref{easy3}.  Let $e_p:MAX(\Omega_{\xi, p})\to MAX(\Omega_{\xi, n})$ be such that for any $\tau\in MAX(\Omega_{\xi, p})$, $f_p(\tau)\preceq e_p(\tau)$, where we omit this step if $p=0$.   Let $\theta:\Omega_{\xi, p}\to \Omega_{\xi, n}$ be given by $\theta''\circ \phi_p$ and let $e:MAX(\Omega_{\xi, p})\to MAX(\Omega_{\xi, n})$ be given by $e=f''\circ e_p$.  Fix any $e_q:MAX(\Omega_{\xi, q})\to MAX(\Omega_{\xi, q})$ such that for any $\tau\in MAX(\Omega_{\xi, q})$, $f_q(\tau)\preceq e_q(\tau)$.      Define $\phi=\theta''\circ \phi_q$, $f'=f''\circ f_q$, and $e'= f''\circ e_q$.

\end{proof}


\begin{thebibliography}{HD}

\normalsize
\baselineskip=17pt

 


\bibitem{AJO} D. Alspach, R. Judd, E. Odell. \emph{The Szlenk index and local $\ell_1$-indices}, Positivity, 9 (1) (2005), 1-44. 

\bibitem{AMT}  S. A. Argyros, S. Mercourakis, and A. Tsarpalias.  \emph{Convex unconditionality and summability of weakly null sequences}, Israel J. Math. 107 (1998), 157-193. 

\bibitem{BCFW} K. Beanland, R.M. Causey, D. Freeman, B. Wallis, \emph{Classes of operators determined by ordinal indices}, J. Funct. Anal. 271 (6) (2016), 1691-1746.  

\bibitem{BP} C. Bessaga, A. Pe\l czy\'{n}ski, \emph{Spaces of continuous functions IV}, Studia Math. 19 (1960) 53-62.

\bibitem{BrookerDirect} P.A.H. Brooker, \emph{Direct sums and the Szlenk index},. J. Funct. Anal. 260 (2011) 2222-2246. 

\bibitem{BrookerAsplund} P.A.H. Brooker, \emph{Asplund operators and the Szlenk index}, Operator Theory 68 (2012), 405-442.

\bibitem{Brooker} P.A.H. Brooker, \emph{Szlenk and $w^*$-dentability indices of the Banach spaces $C([0, \alpha])$}, J. Math. Anal. Appl. 399 (2013), 559-564.


\bibitem{Bo} J. Bourgain. \emph{On convergent sequences of continuous functions}, Bull. Soc. Math. Bel. 32 (1980), 235-249.

\bibitem{Causey} R.M. Causey, \emph{An alternate description of the Szlenk index with applications},  Illinois J. Math.   59 (2) (2015), 359-390. 

\bibitem{Causey1} R.M. Causey, \emph{Proximity to $\ell_p$ and $c_0$ in Banach spaces},  J. Funct. Anal. 269 (12) (2015), 3952-4005. 

\bibitem{DGZ} R. Deville, G. Godefroy, V. Zizler, \emph{Smoothness and renormings in Banach spaces}, Pitman Monographs and Surveys in Pure and Applied Mathematics,
vol. 64, Longman Scientific \& Technical, Harlow (1993). 

\bibitem{Grothendieck} A. Grothendieck, \emph{Produits tensoriels topologiques et espaces nucl\'{e}aires}, Mem. Amer. Math. Soc. 16 (1955).


\bibitem{HL} P. H\'{a}jek and G. Lancien, \emph{Various slicing indices on Banach spaces}, Mediterr. J. Math. 4 (2007) 179-190.



\bibitem{HLM} P. H\'{a}jek, G. Lancien, V. Montesinos, \emph{Universality of Asplund spaces}, Proc. Amer. Math. Soc. 135 (2007), 2031-2035.

\bibitem{JuddOdell} R. Judd, E. Odell. \emph{Concerning Bourgain's $\ell_1$ index of a Banach space}, Israel J. Math 108 (1998) 145-171. 

\bibitem{KOS} H. Knaust, E. Odell, Th. Schlumprecht. \emph{On asymptotic structure, the Szlenk index and UKK properties in Banach spaces}, Positivity 3 (1999), 173-199.

\bibitem{Lancien} G. Lancien, \emph{On the Szlenk index and the weak$^*$-dentability index}, Quarterly J. Math. Oxford 47 (1996) 59-71.

\bibitem{Lancien2} G Lancien, \emph{A survey on the Szlenk index and some of its applications}, RACSAM Rev. R. Acad. Cienc. Exactas F'is. Nat. Ser. A Mat. 100 (1,2) (2006), 209-235.

\bibitem{LPR} G. Lancien, A. Proch\'{a}zka, M. Raja, \emph{Szlenk indices of convex hulls}, J. Funct. Anal. 272 (2017), 498-521. 

\bibitem{Milutin} A.A. Milutin, \emph{Isomorphisms of spaces of continuous functions on compacta of cardinality continuum},  Teoria Funktsii, Funktsional'nyj Analizi ego Prilo\v{z}enija 2 (1966), 150-156.

\bibitem{Monk} J.D. Monk. \emph{Introduction to set theory}, McGraw-Hill, (1969). 

\bibitem{NP} I. Namioka, R. R. Phelps, \emph{Banach spaces which are Asplund spaces}, Duke Math. J. 42 (4) (1975), 735-750.

\bibitem{Nyikos} P. Nyikos, \emph{Various topologies on trees},  Proceedings of the Tennessee Topology Conference (1997) 167-198. 

\bibitem{Ryan} R. Ryan,  \emph{Introduction to Tensor Products of Banach Spaces}, Springer, London, (2002).

\bibitem{Samuel} C. Samuel, \emph{Indice de Szlenk des C(K), S\'{e}minaire de G\'{e}om\'{e}trie des espaces de Banach}, Vol. I-II, Publications Math\'{e}matiques de l'Universit\'{e} Paris VII, Paris (1983), 81-91.


\bibitem{Schlumprecht} Th. Schlumprecht, \emph{On Zippin's embedding theorem of Banach spaces into Banach spaces with bases}, to appear in Advances in Mathematics,  arXiv:1408.3311.  

\bibitem{Szlenk} W. Szlenk, \emph{The non existence of a separable reflexive Banach space universal for all separable reflexive Banach spaces}, Studia Math. 30 (1968), 53-61. 


\end{thebibliography}
\end{document}